\documentclass[12pt]{amsart}
\usepackage{amssymb}
\usepackage{verbatim}
\usepackage{graphicx}
\usepackage[cmtip,all]{xy}
\usepackage{color}
\newtheorem{theo}{Theorem}[section]
\newtheorem{thm}[theo]{Theorem}
\newtheorem{lem}[theo]{Lemma}
\newtheorem{cor}[theo]{Corollary}

\newtheorem{prop}[theo]{Proposition}

\theoremstyle{definition}
\newtheorem{dfn}[theo]{Definition}

\theoremstyle{remark}

\numberwithin{equation}{section}

\newcommand{\complex}{\mathbb{C}}
\newcommand{\su}{\sqsubset}
\newcommand{\C}{\mathbb{C}}
\newcommand{\sphere}{\complex^*}
\newcommand{\disk}{\mathbb{D}}
\newcommand{\bbd}{\mathbb{D}}
\newcommand{\reals}{\mathbb{R}}

\newcommand{\ints}{\mathbb{Z}}
\newcommand{\vp}{\varphi}
\newcommand{\ol}{\overline}
\newcommand{\sm}{\setminus}
\newcommand{\A}{\mathcal{A}}

\newcommand{\ha}{\hat{a}}
\newcommand{\hb}{\hat{b}}
\newcommand{\bd}{\mathrm{Bd}}
\newcommand{\Bd}{\mathrm{Bd}}

\newcommand{\lam}{\mathcal{L}}

\newcommand{\bc}{\bar{c}}
\newcommand{\ba}{\bar{a}}

\newcommand{\diam}{\mathrm{diam}}
\newcommand{\ch}{\mathrm{CH}}
\newcommand{\si}{\sigma}
\newcommand{\0}{\emptyset}
\newcommand{\uc}{\mathbb{S}}
\newcommand{\ucirc}{\mathbb{S}}

\newcommand{\dml}{\mathrm{DML}}
\newcommand{\qml}{\mathrm{QML}}

\newcommand{\g}{\mathfrak{g}}
\newcommand{\h}{\mathfrak{h}}

\newcommand{\e}{\varepsilon}

\begin{document}

\title{Laminations in the language of leaves}

\author[Blokh]{Alexander M. Blokh}

\email{ablokh@math.uab.edu}
\author[Mimbs]{Debra Mimbs}
\email{dmimbs@uab.edu}
\author[Oversteegen]{Lex G. Oversteegen}
\email{overstee@math.uab.edu}
\address[A.~Blokh, D.~Mimbs and L.~Oversteegen]
{Department of Mathematics\\ University of Alabama at Birmingham\\
Birmingham, AL 35294-1170, USA}

\author[Valkenburg]{Kirsten I. S. Valkenburg}
\address{Faculteit der Exacte Wetenschappen, Afdeling Wiskunde, Vrije Universiteit, De Boelelaan 1081a,
1081 HV Amsterdam, The Netherlands} \email{kirstenvalkenburg@gmail.com}
\curraddr{Department of Mathematics and Statistics, University of Saskatchewan, 106 Wiggins Road, Saskatoon, SK, S7N 5E6, Canada}
\thanks{The first and third named authors were supported in part by NSF-DMS-0901038
and NSF-DMS-0906316. The last named author was supported by the
Netherlands Organization for Scientific Research (NWO), under grant
613.000.551; she also thanks the Department of Mathematics at UAB
for its hospitality.}


\date{January 18, 2011; in the revised form December 13, 2011}

\begin{abstract}
Thurston defined invariant laminations, i.e. collections of chords
of the unit circle $\uc$ (called \emph{leaves}) that are pairwise
disjoint inside the open unit disk and satisfy a few dynamical
properties. 
To be directly associated to a polynomial, a lamination has to
be generated by an equivalence relation with specific properties on $\uc$;
then it is called a \emph{q-lamination}. Since not all laminations
are q-laminations, then from the point of view of studying
polynomials the most interesting are those of them which are limits
of q-laminations. In this paper we introduce an alternative
definition of an invariant lamination, which involves only
conditions on the leaves (and avoids gap invariance). The new class
of laminations is slightly smaller than that defined by Thurston and
is closed. We use this notion to elucidate the connection between
invariant laminations and invariant equivalence relations on $\uc$.
\end{abstract}

\maketitle

\section{Introduction}

Invariant laminations, introduced by Thurston in the early 1980's,
are used to study the dynamics of individual polynomials and the
parameter space of all polynomials, the latter in the quadratic case
(an expanded version of Thurston's preprint recently appeared
\cite{thu09}). Investigating the space of all quadratic invariant
laminations played a crucial role in \cite{thu09}. An important idea
of Thurston's was, as we see it, similar to one of the main ideas of
dynamics as a whole - to suggest a tool (\emph{laminations})
allowing one to model the dynamics under investigation on a
topologically/combinatorially nice object (in case of \cite{thu09}
one models polynomial dynamics on the Julia set by so-called
\emph{topological polynomials}, generated by laminations).

According to Thurston, a \emph{lamination} $\lam$ is a closed family of
chords inside the open unit disk $\disk$. These chords meet at most in a common endpoint
and satisfy some dynamical conditions; these chords are
called \emph{leaves (of the lamination)} and union of all
leaves from $\lam$ united with $\uc$ is denoted by $\lam^*$. A natural
\emph{direct} way to associate a lamination to a polynomial
$P$ of degree $d$ with a locally connected Julia set is as follows: (1)
define an equivalence relation $\sim_P$ on $\uc$ by identifying angles
if their external rays land at the same point (observe that $\sim_P$ on
$\uc$ is $\si_d$-invariant); (2) consider the edges of convex hulls of
equivalence classes and declare them to be the \emph{leaves}
of the corresponding \emph{lamination} $\lam_P$.

By \cite{kiw04, bco08} more advanced methods allow one to associate a
lamination to some polynomials with non-locally connected Julia sets
(by declaring two angles equivalent if impressions of
their external rays are non-disjoint and extending this relation by
transitivity). We call laminations, generated by equivalence relations
similar to $\sim_P$ above, \emph{q-laminations}. They form an important
class of laminations, many of which correspond to complex polynomials
with connected Julia sets. In all these cases the lamination is found
through the study of the topology of the Julia set of the polynomial.

The drawback of this approach is that it fails if the topology of the
Julia set is complicated (e.g., if a quadratic polynomial has a fixed
\emph{Cremer point} \cite{bo06}). Thus, even though ultimately
laminations are a tool which allows one to study both individual
polynomials and their parameter space, in some cases it is not obvious as to
what laminations (or what equivalence relations on the circle) can be
directly connected in a meaningful way to certain polynomials. Hence
one needs a non-direct way of associating a lamination (or, more
generally, some combinatorial structure) to a polynomial with a
complicated Julia set.

A possibility here is as follows. For a polynomial $P_c(z)=z^2+c$,
consider sequences of parameters $c_i\to c$ with $P_{c_i}=P_i$ having
locally connected Julia sets and associated lamination $\lam_{P_i}$.
These laminations $\lam_{P_i}$ (systems of chords of $\uc$) may
converge to another lamination (system of chords of $\uc$) in the sense
that the continua $\lam^*_{P_i}$ may converge to a subcontinuum of
$\disk$ in the Hausdorff sense, and the limit continuum $\lam^*$ then
comes from an appropriate lamination $\lam$). In this case the
lamination $\lam$ is called the \emph{Hausdorff limit} of laminations
$\lam_{P_i}$; one may associate all such Hausdorff limit laminations to
$c$.

Using this notion of convergence
one can define the Hausdorff closures of sets of laminations. Hence the
space of laminations useful for studying polynomials could be
a closed set of laminations which contains the Hausdorff closure of the
set of all q-laminations, but is not much bigger.

To describe a candidate set of laminations we introduce a new notion of a \emph{sibling
invariant lamination} which is slightly more restrictive than the
one given by Thurston. The new definition is given intrinsically
(i.e., by only listing properties on the leaves of the lamination). We show that the
family of all sibling invariant laminations is closed and contains
all q-laminations. The new definition significantly simplifies the
verification of the fact that a system of chords of $\uc$ is an
invariant lamination. Thurston \cite{thu09} introduced the class of
\emph{clean} laminations. We use our tools to show that clean laminations are (up to
a finite modification) $q$-laminations. In Section 6 we apply
these ideas to the degree $2$ case and show that in this case all
clean Thurston invariant laminations are $q$-laminations.

\smallskip

\noindent{\textbf{Acknowledgments}}. The authors would like to thank
the referee for useful suggestions and comments.

\section{Laminations: classical definitions}

\subsection{Preliminaries}

Let $\complex$ be the complex plane, $\mathbb{S}\subset\complex$  the unit circle identified with
$\reals/\ints$ and let $\disk\subset \complex$ be the open unit
disk. Define a map $\sigma_d:\mathbb{S} \rightarrow \mathbb{S}$ by
$\sigma_d(z) = dz$ mod 1, $d \geq 2$. By a \emph{chord} in the unit disk we
mean a segment of a straight line connecting two points of the unit circle.
A {\em prelamination}
$\mathcal{L}$ is a collection of chords in $\mathbb{D}$, called {\em
leaves}, such that any two leaves of $\mathcal{L}$ meet at most in a
point of $\uc$.  If all points of the circle are elements of $\lam$
(seen as degenerate leaves) and $\bigcup\mathcal{L}=\lam^*$ is
closed in $\complex$, then we call $\mathcal{L}$ a {\em lamination}.
Hence, one obtains a lamination by closing a prelamination and adding all points of $\uc$ viewed as degenerate leaves. If
$\ell\in\mathcal{L}$ and $\ell\cap\mathbb{S} = \{a,b\}$ then we
write $\ell = \overline{ab}$. 
We use the term ``leaf'' to refer to a non-degenerate leaf in the
lamination, and specify when a leaf may be degenerate, i.e. a point in
$\mathbb{S}$.

Given a leaf $\ell =\overline{ab} \in \mathcal{L}$, let
$\sigma_d(\ell)$ be the chord with endpoints $\sigma_d(a)$ and
$\sigma_d(b)$.  If $\sigma_d(a) = \sigma_d(b)$, call $\ell$ a {\em
critical leaf} and $\sigma_d(a)$ a {\em critical value}.  Let
$\sigma_d^{\ast}:\mathcal{L}^{\ast}\rightarrow\ol{\mathbb{D}}$ be the
linear extension of $\sigma_d$ over all the leaves in $\mathcal{L}$. It
is not hard to check that $\sigma_d^{\ast}$ is continuous. Also,
$\sigma_d$ is locally one-to-one on $\mathbb{S}$, and $\sigma_d^{\ast}$
is one-to-one on any given non-critical leaf.  Note that if
$\mathcal{L}$ is a lamination, then $\mathcal{L}^{\ast}$ is a
continuum.

\begin{dfn}[Gap] \label{dfn-gap}
A {\em gap} $G$ of a lamination $\mathcal{L}$ is the closure of a
component of $\mathbb{D}\setminus\mathcal{L}^{\ast}$; its boundary
leaves are called \emph{edges (of a gap)}. We also say that a leaf
$\ell$ is an \emph{edge} of $\ell$.
\end{dfn}

For each set $A\subset \ol{\disk}$ we denote $A \cap \uc$ by
$\partial(A)$. If $G$ is a leaf or a gap  of $\lam$, it follows that
$G$ coincides with the convex hull of $\partial(G)$. If $G$ is a
leaf or a gap of $\lam$ we let $\sigma_d(G)$ be the convex hull of
$\sigma_d(\partial(G))$. Also, by $\Bd(G)$ we denote the topological
boundary of $G$. Notice that the topological boundary of $G$ is a
Jordan curve which consists of leaves and points on $\mathbb{S}$, so
that $\Bd(G) \cap \mathbb{S} = G \cap \mathbb{S} =\partial(G)$. A
gap $G$ is called \emph{infinite} if and only if $\partial(G)$ is
infinite. A gap $G$ is called \emph{critical} if $\si_d|_{\partial
G}$ is not one-to-one. Observe that there are two types of
degenerate leaves of $\mathcal{L}^{\ast}$ which are not endpoints of
non-degenerate leaves: (1) certain vertices of gaps, (2) points of
$\uc$, separated from other points of $\mathbb{S}$ by a sequence of
leaves of $\mathcal{L}$.

\subsection{q-laminations}\label{qlam}

Let $P$ be a complex polynomial with locally connected Julia set $J$.
Then $J$ is connected and there exists a conformal map $\vp:\sphere\sm
\ol{\disk} \to \sphere\sm K$, where $K$ is the \emph{filled-in Julia
set} (i.e., the complement of the unbounded component of $J$ in
$\complex$). One can choose $\vp$ so that $\vp'(0)>0$ and $P\circ \vp=\vp\circ \si_d$,
where $\si_d(z)=z^d$ and $d$ is the degree of $P$. Since $J$ is locally
connected, $\vp$ extends over the boundary $\uc$ of $\disk$. We denote
the extended map also by $\vp$. Define an equivalence relation
$\approx_P$ on $\uc$ by $x \approx_P y$ if and only if $\vp(x)=\vp(y)$.
Since $J$ is the boundary of $\sphere\sm K$, then $J$ is homeomorphic to $\uc/{\approx_P}$. Clearly, the map $\si_d$
induces a map $f_d:\uc/{\approx_P}\to \uc/{\approx_P}$ and the maps $P|_J$
and $f_d$ are conjugate. It is known that all equivalence classes of
$\approx_P$ are finite. The collection of boundary edges of convex hulls of all
equivalences classes of $\approx_P$ is a lamination denoted by
$\lam_P$.

Equivalence relations analogous to $\approx_P$ can be introduced
with no reference to polynomials \cite{bl02}. Let $\sim$ be an
equivalence relation on $\ucirc$. Equivalence classes of $\sim$ will
be called \emph{($\sim$-)classes} and will be denoted by Gothic
letters. Also, given a closed set $A\subset\C$, let $\ch(A)$ denote
the convex hull of the set $A$ in $\C$.

\begin{dfn}\label{d-lameq}
An equivalence relation $\sim$ is a \emph{($d$-)invariant
laminational} equivalence relation if:

\noindent (E1) $\sim$ is \emph{closed}: the graph of $\sim$ is a closed
set in $\ucirc \times \ucirc$;

\noindent (E2) $\sim$ 
is \emph{unlinked}: if $\g_1$ and $\g_2$ are distinct $\sim$-classes,
then their convex hulls $\ch(\g_1), \ch(\g_2)$ in the unit disk $\bbd$
are disjoint,

\noindent (D1) $\sim$ is \emph{forward invariant}: for a class $\g$,
the set $\si_d(\g)$ is a class too

\noindent which implies that

\noindent (D2) $\sim$ is \emph{backward invariant}: for a class $\g$,
its preimage $\si_d^{-1}(\g)=\{x\in \ucirc: \si_d(x)\in \g\}$ is a
union of classes;

\noindent (D3) for any $\sim$-class $\g$ with more than two points, the
map $\si_d|_{\g}: \g\to \si_d(\g)$ is a \emph{covering map with
positive orientation}, i.e., for every connected component $(s, t)$ of
$\ucirc\setminus \g$ the arc in the circle $(\si_d(s), \si_d(t))$ is a
connected component of $\ucirc\setminus \si_d(\g)$;

\noindent (D4) all $\sim$-classes are finite.
\end{dfn}

There is an important connection between laminations and
(invariant laminational) equivalence relations.

\begin{dfn}\label{d-lamtoeq}\label{d:fineq}
Let $\lam$ be a lamination. Define the equivalence relation
$\approx_\lam$ by declaring that $x{\approx_\lam}y$ if and only if
there exists a finite concatenation of leaves of $\lam$ joining $x$
to $y$.
\end{dfn}

Now we are ready to define \emph{q-laminations}.

\begin{dfn}\label{d-qlam}
A lamination $\lam$ is called a \emph{q-lamination} if the
equivalence relation $\approx_\lam$ is an invariant laminational
equivalence relation and $\lam$ consists exactly of boundary edges of the
convex hulls of all $\approx_\lam$-classes together with all points
of $\uc$. If an invariant laminational equivalence relation $\sim$
is given and $\lam$ is formed by all edges from the convex hulls of
all $\sim$-classes together with all points of $\uc$ then $\lam$ is
called the \emph{q-lamination (generated by $\sim$)} and is denoted
by $\lam_\sim$. Clearly, if $\lam$ is a q-lamination, then it is a
q-lamination generated by $\approx_\lam$.
\end{dfn}

Let $\sim$ be a laminational equivalence relation and $p: \ucirc\to
J_\sim=\ucirc/{\sim}$ be the quotient map of $\ucirc$ onto its
quotient space $J_\sim$, let $f_\sim:J_\sim \to J_\sim$ be the  map
induced by $\sigma_d$. We call $J_\sim$ a \emph{topological Julia
set} and the induced map $f_\sim$ a \emph{topological polynomial}.
It is easy to see from  definition~\ref{d-lameq} that leaves of
$\lam_\sim$ map to leaves of $\lam_\sim$ under $\si_d$; moreover,
the map $\si_d$ acting on leaves and gaps of $\lam_\sim$ has also
other more specific properties analogous to (D1) - (D4) above. This
leads to the abstract notion of an invariant lamination \cite{thu09}
that allows for laminations which are not directly associated to a
laminational equivalence relation and, hence, do not correspond
(directly) to a polynomial.

\subsection{Invariant laminations due to Thurston}


\begin{dfn}[Monotone Map]\label{monotone}
Let $X$, $Y$ be topological spaces and $f:X\rightarrow Y$ be
continuous. Then $f$ is said to be {\em monotone} if $f^{-1}(y)$ is
connected for each $y \in Y$. It is known that if $f$ is monotone and
$X$ is a continuum then $f^{-1}(Z)$ is connected for every connected
$Z\subset f(X)$.
\end{dfn}

Definition~\ref{dfn-Thurston} is due to Thurston \cite{thu09}; recall
that gaps are defined in Definition~\ref{dfn-gap}.

\begin{dfn}[Thurston Invariant Lamination \cite{thu09}] \label{dfn-Thurston}
A lamination $\mathcal{L}$ is {\em Thurston $d$-invariant} if it
satisfies the following conditions.

\begin{enumerate}

\item Forward $d$-invariance: for any leaf $\ell=\overline{pq} \in
    \mathcal{L}$, either $\sigma_d(p) = \sigma_d(q)$, or
    $\overline{\sigma_d(p)\sigma_d(q)}=\si_d(\ell) \in \mathcal{L}$.

\item Backward invariance: for any leaf $\overline{pq} \in
    \mathcal{L}$, there exists a collection of $d$ {\bf disjoint}
    leaves in $\mathcal{L}$ (this collection of leaves may not be
    unique), each joining a pre-image of $p$ to a pre-image of $q$.

\item Gap invariance: For any gap $G$, the convex hull $H$ of $\si_d(G\cap\uc)$
 is a gap, a leaf, or a single point (of $\uc$).







\end{enumerate}

If $H$ is a gap, $\si_d^*|_{\Bd(G)}:\Bd(G)\to\bd(H)$ must map as the
composition of a monotone map and a covering map to the boundary of
the image gap, with positive orientation (the image of a point
moving clockwise around $\bd(G)$ must move clockwise around the image $\bd(H)$
of $G$).
\end{dfn}


\section{Sibling Invariant Laminations}


\subsection{An alternative definition}

Note that in Definition~\ref{dfn-invariant} we do not require the
invariance of gaps.

\begin{dfn}[Sibling $d$-Invariant Lamination \cite{mim10}] \label{dfn-invariant}
A (pre)la-mination $\lam$ is {\em sibling $d$-invariant} if:

\begin{enumerate}

\item for each $\ell\in\mathcal{L}$ either
    $\sigma_d(\ell)\in\mathcal{L}$ or $\sigma_d(\ell)$ is a point
    in $\mathbb{S}$,

\item for each $\ell\in\mathcal{L}$ there exists a leaf
    $\ell'\in\mathcal{L}$  such that $\sigma_d(\ell')=\ell$,

\item \label{disjoint} for each $\ell\in\mathcal{L}$ such that
    $\sigma_d(\ell)$ is a non-degenerate leaf, there exist $\mathbf d$ {\bf
    disjoint} leaves $\ell_1, \dots, \ell_d$ in $\mathcal{L}$ such
    that $\ell=\ell_1$ and $\sigma_d(\ell_i) = \sigma_d(\ell)$ for
    all $i$.

\end{enumerate}

\end{dfn}

We need to make a few remarks. Given a continuum or a point $K\subset \lam^*$
which maps one-to-one onto $\si_d^*(K)$, we call a continuum or a point $T\subset
\lam^*$ a \emph{sibling (of $K$)} if $K\cap T=\0$ and $T$ maps onto $\si_d^*(K)$ in a
one-to-one fashion too (thus, siblings are homeomorphic and disjoint). E.g., the
leaves $\ell_2, \dots, \ell_d$ from Definition~\ref{dfn-invariant} are
siblings of $\ell$. The collection $\{\ell, \ell_2, \dots,
\ell_d\}$ of leaves from Definition~\ref{dfn-invariant} is called a
\emph{full sibling collection (of $\ell$)}. In general, for a
continuum or a point $K\subset \lam^*$ which maps one-to-one onto $\si_d^*(K)$, a
collection of $d$ sets made up of $K$ and its pairwise disjoint
siblings is called a \emph{full sibling collection (of $K$)}.
Often we talk about siblings without assuming the
existence of a full sibling collection (e.g., in the context
of Thurston $d$-invariant laminations).

Let $\lam$ be a sibling $d$-invariant  lamination. Then by
Definition~\ref{dfn-invariant} (1) we see that
Definition~\ref{dfn-Thurston} (1) is satisfied. Now, let $\ell\in
\lam$ be a leaf. Then by Definition~\ref{dfn-invariant} (2) and (3)
there are $d$ pairwise disjoint leaves of $\lam$ which map onto
$\ell$; thus, Definition~\ref{dfn-Thurston} (2) is satisfied.
Therefore both sibling $d$-invariant  laminations and Thurston
$d$-invariant laminations satisfy conditions (1) and (2) of
Definition~\ref{dfn-Thurston}, i.e. are \emph{forward $d$-invariant}
and \emph{backward $d$-invariant}. Both conditions deal with
\emph{leaves} and in that respect are intrinsic to $\lam$ which is
defined as a collection of leaves. However having these conditions
is not enough to define a meaningful dynamic collection of leaves;
there are examples of laminations satisfying conditions (1)-(2) of
Definition~\ref{dfn-Thurston} which are not gap invariant. Therefore
one needs to add an extra condition to forward and backward
invariance.

The choice made in Definition~\ref{dfn-Thurston} deals with gaps,
i.e. closures of components of the complement $\disk\sm \lam^*$.
This is a straightforward way to ensure that $\si^*_d$ has a nice
extension over the plane. However a drawback of this approach is
that while $\lam$ otherwise is defined as a family of chords of
$\disk$ (leaves), in gap invariance we directly talk about other
objects (gaps). One can argue that gap invariance of $\lam$ under
$\si_d$ is not sufficiently intrinsic since $\lam$  is defined as a
collection of leaves. As a consequence it is often more cumbersome
to verify gap invariance. This justifies the search for a similar
definition of an invariant lamination which deals only with leaves.
Above we propose the notion of a \emph{sibling ($d$)-invariant
lamination}.

\subsection{Sibling invariant laminations are gap invariant}\label{gapinv1}

Now we show that any sibling $d$-invariant lamination is a Thurston
$d$-invariant lamination. Some complications below are caused by the fact that we
do not yet know that the lamination is gap invariant. E.g., extending
the map $\si^*$ over $\disk$ to a suitably nice map (i.e., the
composition of a monotone and open map) is impossible if the lamination
is not gap invariant.




\begin{thm} \label{gapinv}
Suppose that $G$ is a gap of a sibling $d$-invariant lamination
$\mathcal{L}$.  Then either

\begin{enumerate}

\item $\sigma_d(G)$ is a point in $\mathbb{S}$ or a leaf of
    $\mathcal{L}$, or

\item there is a gap $H$ of $\mathcal{L}$ such that $\sigma_d(G) =
    H$, and the map $\sigma_d^{\ast}|_{\Bd(G)}: \Bd(G) \rightarrow \Bd(H)$ is
    the positively oriented
  composition of a monotone map $m: \Bd(G) \rightarrow S$, where
  $S$ is a simple closed curve, and a covering map $g: S
  \rightarrow \Bd(H)$.
\end{enumerate}

Thus, any sibling $d$-invariant lamination is a Thurston $d$-invariant lamination.
\end{thm}

To prove Theorem~\ref{gapinv} we prove a few lemmas.
Given a point $x$, call any point $\hat x\in \mathbb{S}$ with $\sigma_d(\hat x)=x$
an \emph{$x$-point}.
If a lamination is given, by an $\ha\hb$-{\em leaf}, we mean a leaf
that maps to $\ol{ab}$. The word ``chord" is used in lieu of ``leaf"
in reference to a chord of $\disk$ which may not be a leaf of
$\lam$. By $[a, b]_\uc$, $a, b\in \uc$ we mean the closed arc of
$\uc$ from $a$ to $b$, and by $(a,b)_\uc$ we mean the open arc of
$\uc$ from $a$ to $b$. The direction of the arc, clockwise
(\emph{negative}) or counterclockwise (\emph{positive}), will be
clear from the context; also, sometimes we simply write $[a, b], (a,
b)$ etc.  By $<$ we denote the positive (circular) order on $\uc$.
If we say that points are \emph{ordered} on $\uc$ we mean that they
are either positively or negatively ordered.
Proposition~\ref{abcount} is left to the reader; observe, that in
Proposition~\ref{abcount} we do not assume the existence of a lamination.



\begin{prop}\label{abcount}
Suppose that $a_1<b_1<a_2<b_2<\dots<a_n<b_n$ are $2n$ points in the circle.
Then for any point $a_i$ and $b_j$ either component of
$\uc\sm\{a_i, b_j\}$ contains the same number of $a$-points and $b$-points.
In particular, if $\ha, \hb\in \uc$ are such that $a=\sigma_d(\ha) \neq
\sigma_d(\hb)=b$, then either component of $\uc\sm \{\ha, \hb\}$
contains the same number of $a$-points and $b$-points.
\end{prop}

Since $2$-invariant laminations are invariant under the rotation by
$\frac{1}{2}$, then, given a $2$-invariant lamination we see that its siblings are rotations
of each other. Even though this is not typically
true for laminations of higher degree (see Figure~\ref{ocon}),
Lemma~\ref{posor} states that sibling leaves must connect in the
same order. To state it we need Definition~\ref{orient}.

\begin{figure}[h!]
\centering\def\svgwidth{.61\columnwidth}\fbox{\fbox{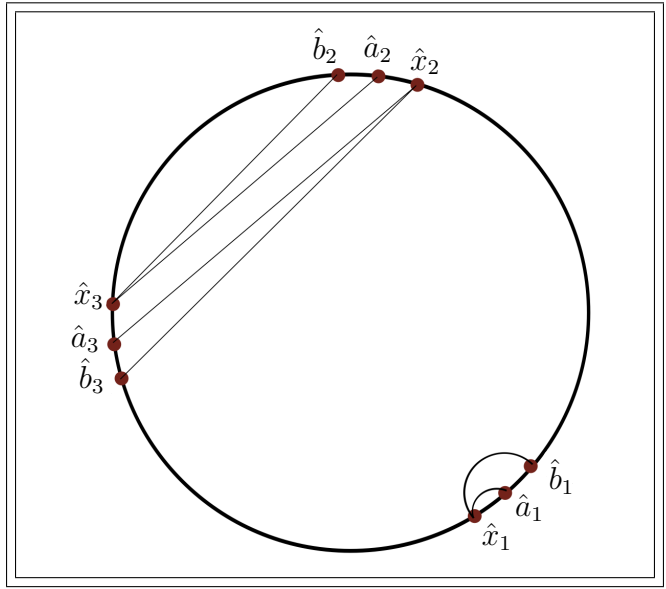}} 
 \caption{Sibling ``arcs"} This is an example of
sibling ``arcs" under $\sigma_3$.  Notice that while the arcs
connect points in different ``patches" and are not found by rigid
rotation, the manner in which the endpoints connect preserves order.
\label{ocon}
\end{figure}

\begin{dfn}\label{orient}
Consider two disjoint sets $A, B\subset \uc$ such that $\si_d(A)=\si_d(B)=C$ is
the one-to-one image of $A$ and $B$ under $\si_d$. Then $A, B$ and $C$
are said to have the \emph{same orientation} if for any three points
$x, y, z\in A$ their siblings $x', y', z'\in B$ and their images
$\si_d(x), \si_d(y), \si_d(z)$ have the same circular orientation as
$x, y, z$. As we walk along the circle in the positive direction from a
point $u\in A$, its sibling $u'\in B$, and its image $\si_d(u)$, we
will meet points of $A$, their siblings in $B$, and their images in $C$
in the same order.
\end{dfn}

Any two two-point sets have the same orientation; this is not
necessarily true for sets with more points. The fact that sets have
the same orientation sometimes implies ``structural'' conclusions.
For a set $A\subset \lam^*$ we write $A\su \lam^*$ if $A\cap \uc$ is
zero-dimensional.

\begin{dfn}\label{triod-def}
A \emph{triod} is a homeomorphic image of the \emph{simple triod}
$\tau$ (the union of three arcs which share a common endpoint).
Denote by $B(T)$ the union of the endpoints and the vertex of a
triod $T$. In what follows we \emph{always} consider triods
$T\subset \disk$ with $B(T)\subset \uc$. The edge of $T$, which
separates (inside $\ol{\disk}$) the  endpoints of $T$ non-belonging
to it, is called the \emph{central edge} of $T$ while the other
edges of $T$ are said to be \emph{sides} of $T$. Similarly, if $A$
is the union (the concatenation) of two leaves $\ol{av}\cup \ol{vb}$
we set $B(A)=\{a, v, b\}$.
\end{dfn}

To avoid ambiguity we call a subarc of $\uc$ a \emph{circle arc}. By
an \emph{arc in $\lam^*$} we mean a \emph{topological arc} (a
one-to-one image of an interval). Given a set $A\subset \lam^*$ we
sometimes need to consider a topological arc in $A$ with endpoints
$x, y$ (it will always be clear which arc we actually consider);
such an arc will be denoted by $[x, y]_A$. Thus, $[a, b]_\uc$ is
always a circle arc. By default arcs $[a, b], (a, b)$ etc. are
circle arcs.

\begin{lem} \label{posor}
{\rm (1)} Let $x_1<a_1<b_1<x_2<\cdots<x_n<a_n< b_n < x_1$ be $3n$
points in $\mathbb{S}$. If for each $i$ there exists $r(i), m(i)\in
\{1, \dots, n\}$ such that arcs $A_i=\ol{x_ia_{r(i)}}\cup
\ol{x_ib_{m(i)}}$ are pairwise disjoint ($1\le i\le n$) then $x_i<
a_{r(i)}< b_{m(i)}$ for each $i$.

{\rm (2)} Let $x_1<a_1<b_1<c_1<x_2<\cdots<x_n<a_n<b_n<c_n<x_1$ be
$4n$ points in $\mathbb{S}$. 
If for each $i$ there exist $r(i), m(i), l(i) \in \{1, \dots, n\}$ such
that triods $T_i=\ol{x_ia_{r(i)}}\cup \ol{x_ib_{m(i)}}\cup
\ol{x_ic_{l(i)}}$ are pairwise disjoint ($1\le i\le d$) then $x_i<
a_{r(i)}<b_{m(i)}<c_{l(i)}$ for each $i$.
\end{lem}

\begin{proof}
(1) Let $x_1<b_{m(1)}<a_{r(1)}$. Then the arc $(x_1, b_{m(1)})$
contains $m(1)-1$ points $x_2,$ $\dots,$ $x_{m(1)}$, $m(1)-1$ points
$b_1,$ $\dots,$ $b_{m(1)-1}$, but $m(1)$ points $a_1,$ $\dots,$
$a_{m(1)}$. Clearly, this contradicts the existence of sets $A_j$.

(2) Follows from (1) applied to parts of the triods $T_i$.
\end{proof}

We will mostly apply the following corollary of Lemma~\ref{posor}.

\begin{cor}\label{cor-posor1}
Let $\lam$ be a sibling $d$-invariant lamination and $T\su
\lam^*$ be an arc consisting of two leaves with a common endpoint
$v$ or a triod consisting of three leaves with a common endpoint
$v$. Suppose that $S\su \lam^*$ is an arc (triod) such that
$\si^*_d(S)=T$ and $\si^*_d|_S$ is one-to-one. Then the circular
orientation of the sets $B(T)$ and $B(S)$ is the same.
\end{cor}

\begin{proof}
Let the endpoints of $T$ be $a, b$ ($a, b, c$ if $T$ is a triod). Then
the set of all preimages of points of $B(T)$ consists of $d$ triples
(if $T$ is an arc) or quadruples (if $T$ is a triod) of points denoted
by $B_1, \dots, B_d$ and such that (1) each $B_i$ is contained in a
circle arc $J_i$ so that these arcs are disjoint, (2) the circular
order of points in $B_i$ is the same as the circular order of
$\si_d$-images of these points.

Take the leaves which comprise $T$ (two leaves if $T$ is an arc and
three leaves if $T$ is a triod). Consider the corresponding leaves
comprising $S$. Each leaf of $T$ gives rise to its full sibling
collection (here we use the fact that $\lam$ is sibling invariant).
Then leaves from those collections ``grow'' out of points $v_1, \dots,
v_d$ which are preimages of $v$. This gives rise to $d$ unlinked sets
$S_1, \dots, S_d$ where $S_i$ is a union of two (three) leaves growing
out of $v_i$ (indeed, no leaf of $S_i$ can coincide with a leaf of
$S_j$ where $i\ne j$ while distinct leaves must be disjoint by the
properties of laminations). Moreover, we may assume that $S_1=S$. It
now follows from Lemma~\ref{posor} and the above paragraph that all the
sets $B(S_i)$ of endpoints of $S_i$ united with $x_i$ have the same
circular orientation coinciding with the circular orientation of the
set of their $\si_d$-images, i.e. the set $B(T)$.
\end{proof}

Corollary~\ref{cor-posor1} shows that all pullbacks of certain sets have
the same orientation as the sets themselves. However it also allows us
to study images of some sets. Indeed, by Corollary~\ref{cor-posor1}, if
$S\su \lam^*$ is a triod mapped by $\si^*_d$ one-to-one into $T$
then the central edge of $S$ maps into the central edge of $T$.

Lemmas~\ref{posor} and Corollary~\ref{cor-posor1} are useful in
comparing the orientation of arcs (triods) and their images \emph{in
the absence of critical leaves}. In the case when there are critical
leaves in the arcs and triods we need additional lemmas. In what follows
by a \emph{preimage collection (of a chord $\ol{ab}$}) we mean a
collection $A$ of several \emph{pairwise disjoint} chords with the
same \emph{non-degenerate} image-chord $\ol{ab}$; here we do not
necessarily assume the existence of a lamination. However if we deal
with a lamination $\lam$ then we always assume that preimage
collections consist of leaves of $\lam$ and often call them
\emph{preimage leaf collections}. If there are $d$ disjoint chords
in $A$ we call it \emph{full}. If $X$ is a preimage collection of a
chord $\ol{ab}$, we denote the endpoints of chords of $X$ by the same letters
as for the endpoints of their images but with a hat and distinct
subscripts, and call them correspondingly ($a$-points, $b$-points
etc). Finally, recall that $\partial(X)$ is the union of all
endpoints of chords from $X$.


\begin{lem} \label{loca}
Let $X$ be a full preimage collection of a chord $\ol{ab}$ and
$\ol{\ha_1\hb_1}$, $\ol{\ha_2\hb_2}$ be two chords from $X$. Then
the number of chords from $X$ crossing the chord $\ol{\ha_1\ha_2}$
inside $\disk$ is even if and only if either $\ha_1<\hb_1<\ha_2<\hb_2$ or
$\ha_1<\hb_2<\ha_2<\hb_1$. In particular, if there exists a concatenation
$Q$ of chords connecting $\ha_1$ and $\ha_2$, disjoint  with chords of
$X$ except the points  $\ha_1, \ha_2$, then either $\ha_1<\hb_1<\ha_2<\hb_2$ or
$\ha_1<\hb_2<\ha_2<\hb_1$.
\end{lem}

The fact that either $\ha_1<\hb_1<\ha_2<\hb_2$ or
$\ha_1<\hb_2<\ha_2<\hb_1$ is equivalent to the fact that
$\ol{\ha_1\ha_2}$ separates $\ol{\ha_1\hb_1}\sm \{\ha_1\}$ from
$\ol{\ha_2\hb_2}\sm \{\ha_2\}$ in $\disk$.

\begin{proof}
See Figure~\ref{sibs}. First let us show that if, say,
$\ha_1<\hb_1<\ha_2<\hb_2$ then the number of chords from $X$
crossing the chord $\ol{\ha_1\ha_2}$ inside $\disk$ is even. Indeed,
by Proposition~\ref{abcount} there are, say, $k$ $a$-points and $k$
$b$-points in $(\hb_1, \ha_2)$. Suppose that among chords of $X$
there are $m$ chords with both endpoints in $(\hb_1, \ha_2)$. Then
there are $2k-2m$ $a$- and $b$-points in $(\hb_1, \ha_2)$ which are
exactly all the endpoints of chords from $X$ which cross
$\ol{\ha_1\ha_2}$. inside $\disk$. This implies that the number of
chords from $X$ crossing the chord $\ol{\ha_1\ha_2}$ inside $\disk$
is even.

On the other hand, suppose that the number of chords from $X$
crossing the chord $\ol{\ha_1\ha_2}$ inside $\disk$ is even. As
before, for definiteness assume that $\ha_1<\hb_1<\hb_2<\ha_1$. For
any $Z\subset X$ consider a function $\vp(I, Z)$ of an arc $I\subset
\uc$, defined as the difference between the number of $a$-points in
$Z\cap I$ and the number of $b$-points in $Z\cap I$ taken modulo
$2$. Clearly, for some $k$ the arc $(\hb_1, \hb_2)$ contains $k$\,
$b$-points and $k+1$ $a$-points; similarly, for some $l$ the arc
$(\ha_2, \ha_1)$ contains $l$ $a$-points and $l+1$ $b$-points.
Hence, $\vp((\hb_1, \hb_2), X)=\vp((\ha_2, \ha_1), X)=1$.

\begin{figure}[h!]
\centering\def\svgwidth{.61\columnwidth}\fbox{\fbox{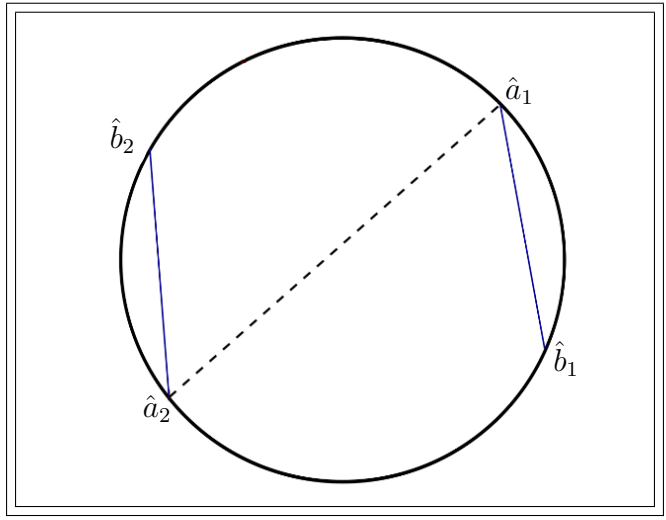}}
\caption{Siblings and critical leaves} Siblings must be on opposite
sides of the chord $\ol{\ha_1\ha_2}$ which is not crossed by leaves from $X$.
\label{sibs}
\end{figure}

Remove the chords from $X$ connecting the arcs $(\hb_1, \hb_2)$ and $(
\ha_2,\ha_1)$ from $X$ to get a new set of chords $Y$. As we
remove one such chord, we increase the value of $\vp$ on $(\hb_1,
\hb_2)$ by $1$, and we increase the value of $\vp$ on $(\ha_2, \ha_1)$
by $1$ as well. By the assumption, to get the set $Y$ we need to remove
an even number of chords. Thus, we see that $\vp((\hb_1, \hb_2),
Y)=\vp((\ha_2, \ha_1), Y)=1$. However, the remaining points of
$\partial(Y)\cap (\hb_1, \hb_2)$ are endpoints of a certain number of
chords from $X$ and hence there must be an equal number of $a$-points
and $b$-points among them, a contradiction.

If there exists a concatenation $Q$ of chords connecting $\ha_1$ to $\ha_2$ so
that $Q\cap X=\{\ha_1,\ha_2\}$ then no chord of $X$ can cross the chord
$\ol{\ha_1\ha_2}$ inside $\disk$. Hence the desired result follows from the first part.
\end{proof}

To prove Lemma~\ref{triorder} we need more definitions.

\begin{dfn}\label{monarc}
Let $I\su \lam^*$ be an arc (the image of a homeomorphism
$\vp:[0, 1]\to I$). We call $I$ a \emph{monotone} arc if its
endpoints $\vp(0), \vp(1)$ belong to $\uc$ and there is a circle arc
$T=[\vp(0), \vp(1)]_\uc$ which contains $\partial(I)$ (this implies
that the map $\vp^{-1}|_{\partial(I)}$ is monotone with respect to
the circular order on $T$). Likewise, a triod $T\su \lam^*$  is
\emph{monotone} if all its edges are monotone arcs.
\end{dfn}

As an example of a monotone arc one can consider a single leaf of
$\lam$ or a subarc of the  boundary  of a gap of $\lam$.

We are ready to prove the following lemma.

\begin{lem}\label{triorder}
Let $\lam$ be a sibling $d$-invariant lamination. Suppose that
$\si^*_d$ monotonically maps a monotone arc $I$ onto the union of
the two sides of a monotone triod $T$ with vertex $v$ whose central
edge is a leaf $\ol{vm}$. Then there exists $\hat v\in \partial(I)$ such
that $\si_d(\hat v)=v$ and there exists a leaf $Q=\ol{\hat v\hat m}$ such that
$I\cup Q$ is a monotone triod with vertex $u$ and central edge $Q$.
\end{lem}

\begin{proof}
Denote the endpoints of $T$ distinct from $m$ by $a$ and $b$. Then
the endpoints of $I$ map to $a$ and $b$; denote them $\ha$ and
$\hb$, respectively. For the sake of definiteness assume that $v\in
(a, b)$. Consider $\si_d|_{(\hb, \ha)}$. Clearly, as we move from
$\hb$ to $\ha$ we first encounter several semi-open subarcs of
$(\hb, \ha)$ which wrap around the circle in the one-to-one fashion.
Then the last arc which we encounter connects a $b$-point with an
$a$-point and maps onto $[b, a]$ in the one-to-one fashion. Hence
there is one more $m$-point in $(\hb, \ha)$ than $v$-points in
$(\hb, \ha)$. This implies that one $m$-point belonging to $(\hb,
\ha)$ (denote it by $\hat m$) must be connected with a leaf to a
$v$-point belonging to $(\ha, \hb)$ (denote it by $\hat v$). This
completes the proof.
\end{proof}



By a \emph{polygon} we mean a finite convex polygon.
In what follows by a \emph{collapsing polygon} we mean a polygon $P$
with edges which are chords of $\disk$ such that their images are the
same \emph{non-degenerate} chord (thus as we walk along the edges of
$P$, their $\si_d$-images walk back and forth along the same
non-degenerate chord). When we say that $Q$ is a \emph{collapsing
polygon of a lamination $\lam$}, we mean that {\bf all} edges of $Q$
are leaves of $\lam$; we also say that $\lam$ \emph{contains a
collapsing polygon $Q$}. However, this does not necessarily imply that
$Q$ is a gap of $\lam$ as $Q$ might be further subdivided by leaves of
$\lam$ \emph{inside} $Q$.

We often deal with \emph{concatenations of leaves}, i.e. finite
collections of pairwise distinct leaves which, when united, form a
topological arc in $\disk$. The concatenation of leaves
$\ell_1,\dots, \ell_k$ is denoted $\ell_1\cdots\ell_k$. If the
leaves are given by their endpoints $x_1,\dots, x_k$, we denote the
concatenation by $\ol{x_1\cdots x_k}$. We do not assume that points
$x_1, \dots, x_k$ are ordered on the circle; however if they are, we
call $\ol{x_1\cdots x_k}$ an \emph{ordered} concatenation.

\begin{lem}\label{colpol}
Suppose that $\lam$ is a sibling $d$-invariant lamination which
contains two distinct leaves $\ell_0=\ol{vx}$ and $\ell_1=\ol{vy}$
such that $\si_d(\ell_0)=\si_d(\ell_1)=\ell$ is a non-degenerate leaf. Then
$\lam$ contains a collapsing polygon $P$ with edges $\ell_0$ and
$\ell_1$ such that $\si_d(P)=\ell$; also, it contains a critical gap
$G\subset P$ with vertex $v$ such that $\si_d(G)=\ell$.
\end{lem}

\begin{proof}
First assume that $x<v<y$ and that there are no leaves
$\ell'=\ol{vz}\in\lam$ with $y<z<x$ and $\si_d(\ell')=\ell$. Since
$\lam$ is a \emph{sibling $d$-invariant lamination}, we can choose a full
sibling collection $A_0\subset \lam$ of $\ell_0$. By
Lemma~\ref{loca} there exists $u_1\in (y, x)$ such that $\ol{yu_1}\in
A_0$ and $\ell_0$ are siblings. Similarly, there exists a full sibling
collection $A_1\subset \lam$ of $\ell_1$ and a point $u_0\in
(y, x)$ such that $\ol{u_0x}\in A_1$ and $\ell_1$ are siblings.
Since $\ol{yu_1}$ and $\ol{u_0x}$ are disjoint inside $\disk$, then
$y<u_1\le u_0<x$.

Consider all possible choices of points $u_0, u_1$ so that the above
properties hold: $\ol{yu_1}$ and $\ell_0$ are siblings, $\ol{u_0x}$
and $\ell_1$ are siblings, and $y<u_1\le u_0<x$. Observe that now we
do not require that $u_0$ or $u_1$ be obtained as endpoints of
siblings of $\ell_0$ or $\ell_1$ coming from full sibling
collections, but the existence of such collections shows that the
set of the choices is non-empty (see the first paragraph). Choose
$u_0, u_1$ so that the arc $[u_1, u_0]$ is the shortest.

If $u_0=u_1$ then we obtain a collapsing polygon $P=\ch(vyu_0x)$.
Suppose that $u_0\ne u_1$. Then by the construction and by the
choice of $u_0$ and $u_1$ no leaf of $\lam$ which maps onto $\ell$
can cross the chords $\ol{vu_0}, \ol{vu_1}$. By Lemma~\ref{loca}
applied to $A_0$ and $\ell_0\in A_0$, and because of the location of
the points found so far, there exists a sibling $\ol{w_0u_0}\in A_0$
of $\ell_0$ with $w_0\in (u_1, u_0)$. Similarly, there exists a
sibling $\ol{u_1w_1}\in A_1$ of $\ell_1$ with $w_1\in (u_1, u_0)$.
Since leaves $\ol{w_0u_0}$ and $\ol{u_1w_1}$ do not intersect inside
$\disk$, we see that $u_1<w_1\le w_0<u_0$. Similar to what we did
before, we can choose $w_1$ and $w_0$ so that $\ol{w_0u_0}$ is a
sibling of $\ell_0$, $\ol{w_1u_1}$ is a sibling of $\ell_1$, and the
arc $(w_1w_0)$ is the shortest possible. We can continue in this
manner and obtain a collapsing polygon $P$ with edges $\ell_0$ and
$\ell_1$.

Now, suppose that there are leaves $\ell'$ between $\ell_0$ and
$\ell_1$ with $\si_d(\ell')=\ell$. Let $K$ be the collection of
\emph{all} such leaves $\ell'$ together with $\ell_0$ and $\ell_1$.
By the above we can form collapsing polygons for each pair of
adjacent leaves from $K$. If we unite them and erase in that union
all leaves of $K$ except for $\ell_0$ and $\ell_1$, we will get a
collapsing polygon $P$ with edges $\ell_0$ and $\ell_1$ (leaves of
$K$ are diagonals of $P$). This proves the main claim of the lemma.
Let $G$ be any gap of $\lam$ contained in $P$ and with edge
$\ol{vx}$. Then $\si^*(G)=\si_d(\ol{vx})$ and, hence, $G$ is
critical.
\end{proof}

We need the following definition.

\begin{dfn}
Given a leaf $\ell = \overline{xy}$, we define the corresponding
{\em open leaf} to be $\ell^{\circ} = \ell \setminus \{x,y\}$. For a
lamination $\lam$, denote its critical leaves by
$\bc_i(\lam)=\bc_i$. Below we often consider the set $\cup_i
\bc^{\circ}_i$ which is the union of \emph{all}  open
critical leaves of $\lam$.
\end{dfn}





\begin{figure}[h!]
\centering\fbox{\fbox{\includegraphics[scale = 0.70]
{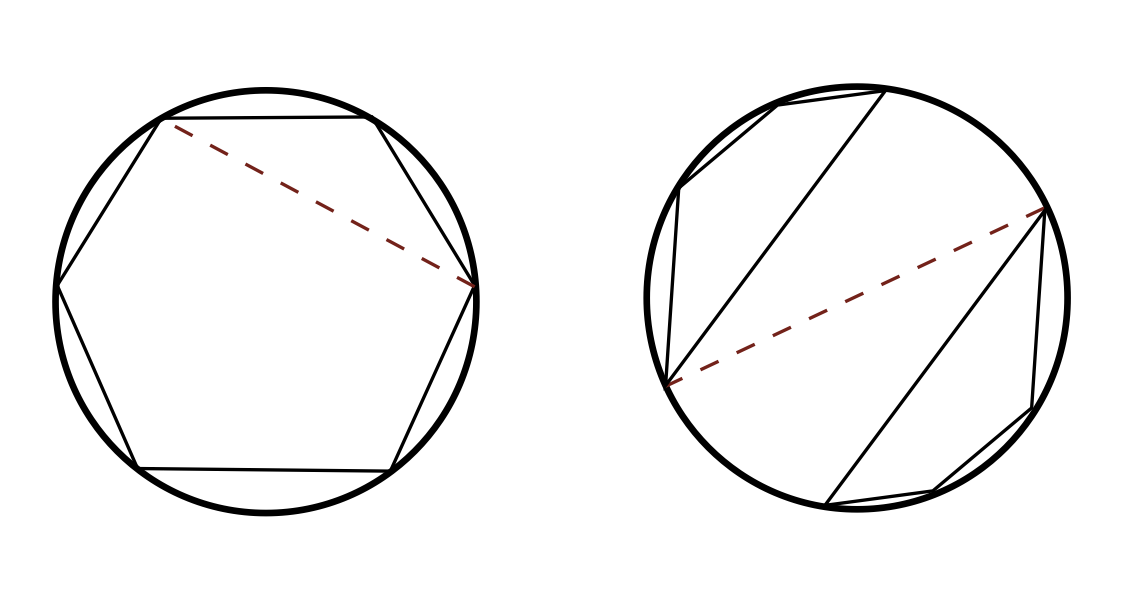}}} \caption{Placements of critical leaves} In each
picture, the critical leaf is denoted by a dashed line. Notice that
in the first example, removing the open critical leaf does not
disconnect the polygon, while in the second example, removing the
open critical leaf disconnects a previously connected set, resulting
in two components.
\end{figure}


\begin{lem}\label{collapse}
Let $\lam$ be a sibling $d$-invariant lamination
and $\ell = \overline{ab} \in \mathcal{L}$ be a leaf.
If $C$ is a component of $\{(\sigma_d^{\ast})^{-1}(\ell)\setminus
\cup_i \bc^{\circ}_i\}$ and $G$ is the convex hull of $\partial(C)$,
then $G$ is a leaf or a collapsing polygon of $\lam$.
\end{lem}


\begin{proof}
If $C$ is not a leaf then there exists $x\in \partial(C)$ which is a
vertex of at least two leaves from $C$. Choose leaves $\ell'$ and
$\ell''$ in $C$ which form an angle containing all other leaves from $C$ with
the endpoint $x$. By Lemma~\ref{colpol} $\ell'$ and $\ell''$ are edges
of a collapsing polygon $P$. By properties of laminations all other
leaves of $C$ with an endpoint $x$ (if any) are diagonals of $P$. Choose
a collapsing polygon $P$, which is maximal by inclusion,
with edges $\ell'$ and $\ell''$.

Let $y$ be a vertex of $P$ and a leaf
$\ol{yz}\subset C$ come out of $y$ and is not contained in $P$. By
properties of a lamination then $\ol{yz}$ is in fact disjoint from $P$
(except for $y$). Choose the edge $\ell'''$ of $P$ with an endpoint $y$
so that the triangle formed by the convex hull of $\ell'''$ and $\ol{yz}$ is disjoint from
the interior of $P$. Then by Lemma~\ref{colpol} there exists a
collapsing polygon $P'$ with edges $\ol{yz}$ and $\ell'''$. It follows
that $P\cup P'$ with $\ell'''$ removed is a collapsing polygon strictly
containing $P$, a contradiction. Thus, $P=G$ as desired.
\end{proof}

From now on by $P_1(\ell), \dots, P_k(\ell)$, with $\ell=\ol{ab}$
non-degenerate, we mean the $\{(\sigma_d^{\ast})^{-1}(\ell)\setminus
\cup_i \bc^{\circ}_i\}$ (leaves or collapsing polygons) from
Lemma \ref{collapse}. Note that all edges of the sets
$P_i(\ell)$ are leaves of $\lam$. We view the set
$(\si^*_d)^{-1}(\ell)$ as follows. By Definition~\ref{dfn-invariant},
there is a full preimage collection $L$ of $\ell$. The endpoints of its
leaves are either $a$-points or $b$-points (depending on whether they
map to $a$ or $b$). Often these are all the leaves mapped to $\ell$.
Yet, there might exist other leaves which map into $\ell$. Some of
these leaves map onto $\ell$; a leaf like that connects an $a$-point
from one leaf of $L$ to a $b$-point from another. Some of these leaves
are critical and map to $a$ ($b$); a leaf like that connects two
$a$-points ($b$-points) belonging to distinct leaves from $L$. Consider
the sets $P_1(\ell), \dots, P_k(\ell)$. There might exist leaves from
$(\si^*_d)^{-1}(\ell)$ \emph{inside} $P_i(\ell)$, however we often
ignore these leaves (which might be either critical or not). There
might also exist other leaves connecting sets $P_i(\ell)$ and
$P_j(\ell)$, $i\ne j$. By Lemma~\ref{collapse} such leaves must be
critical.

\begin{lem}\label{colpol1}
Suppose that $\ol{xv_0\cdots v_ky}=M$ is an ordered concatenation of
leaves of $\lam$ such that $\si_d(v_i)=w$ for all $i$ and
$\si_d(x)=\si_d(y)\neq w$. Then there exists a collapsing polygon
$P_i(\si^*_d(\ol{xv_0}))$ which contains $M$.
\end{lem}

\begin{proof}
We see that $\ol{xv_0}$ and $\ol{v_ky}$ have the same non-degenerate
image while all leaves $\ol{v_0v_1}, \ol{v_1v_2}, \dots$ are critical.
Assume that $x=x_0<v_0<\dots<v_k<y=y_0$. Applying Lemma~\ref{loca} to
the leaf $\ol{x_0v_0}$ and its full sibling collection, we get a point $x_1, v_0<x_1<v_1$ and a leaf
$\ol{x_1v_1}$ which is a sibling of $\ol{x_0v_0}$. Similarly we get
points $x_2, \dots, x_k$ located between points $v_1, v_2,\dots, v_k$
and leaves $\ol{x_iv_i}$ which are siblings of $\ol{x_0v_0}$.

Now, apply Lemma~\ref{colpol} to leaves $\ol{x_kv_k}$ and
$\ol{v_ky_0}$. Then there exists a collapsing polygon $P_0$ with these
leaves as edges. It follows that $v_{k-1}$ is a vertex of $P_0$ and
there is an ordered concatenation of siblings of $\ol{y_0v_k}$ which
begins with $\ol{y_0v_k}$ and ends with some leaf $\ol{y_1v_{k-1}}$. We
can pair this leaf up with the leaf $\ol{x_{k-1}v_{k-1}}$ and apply the
same arguments. In this manner we will discover a ``long'' ordered
concatenation of siblings which begins with $\ol{y_0v_k}$ and ends with
$\ol{v_0x_0}$. By Lemma~\ref{collapse} there exists a collapsing
polygon $P_i(\si^*_d(\ol{xv_0}))$ from the collection of collapsing
polygons described in Lemma~\ref{collapse} which contains this
concatenation. It follows that $M\subset P_i(\si^*_d(\ol{xv_0}))$ as
desired.
\end{proof}

A lamination is called {\em gap-free} if it has no gaps. In the next few
lemmas we study such laminations. Lemma~\ref{wedge} is left to the reader.

\begin{lem}\label{wedge}
$\mathcal{L}$ does not contain a collection of leaves with one common
endpoint such that their other endpoints fill up an arc $I\subset \uc$.
\end{lem}

A continuous interval map $f:I\to I$ is called a \emph{$d$-sawtooth
map} if it has $d$ intervals of monotonicity of length $\frac{1}d$ and
the slope on each such interval is $\pm d$.

\begin{lem}\label{gapfree}
If $\lam$ is gap-free then it consists of a family of pairwise
disjoint parallel leaves which fill up the entire disk $\ol{\disk}$
except for two diametrically opposed points $a, b\in \uc$. The
factor map $p$ which collapses each leaf to a point, semiconjugates
$\si_d$ to a $d$-sawtooth map.
\end{lem}

\begin{proof}
Let $\ell_0, \ell_1\in \lam$ be leaves with a common endpoint. By
Lemma~\ref{wedge} we may assume that $\ell_0, \ell_1$ form a wedge with
no leaves of $\lam$ in it. This implies that $\lam$ is not gap-free, a
contradiction. Hence all leaves are pairwise disjoint. Consider an
equivalence relation on $\disk$ identifying every leaf into one class.
The absence of gaps implies that then the quotient space is an interval
$I$ and that there are only two points $a, b\in \uc$ which are not
endpoints of leaves from $\lam$. Moreover, if $p:\uc\to I$ is the
corresponding factor map then $p(a), p(b)$ are the endpoints of $I$.
Moreover, all leaves of $\lam$ must cross the chord $\ol{ab}$. Indeed,
suppose that $\ol{uv}$ is a leaf of $\lam$ such that the circular arc
$I=[u, v]_{\uc}$ contains neither $a$ nor $b$. Then there must be a gap
of $\lam$ with vertices in $I$ or a point $t\in (u, v)$ disjoint from
all leaves of $\lam$, a contradiction,

Since $\lam$ is invariant, then either $\si_d(a)=a, \si_d(b)=b$, or
$\si_d(a)=b, \si_d(b)=a$, or $\si_d(a)=\si_d(b)=a$, or
$\si_d(a)=\si_d(b)=b$. Consider the case when $\si_d(a)=a, \si_d(b)=b$
(other cases are similar). Then by continuity it follows that as we
travel from $a$ to $b$ along the leaves of $\lam$ we meet critical
leaves $\bc_1, \bc_2, \dots, \bc_{d-1}$ (in this order). Suppose that a
critical leaf $\bc_i$ maps to an endpoint of a leaf $\ell\in \lam$.
Then some preimage-leaf of $\ell$  will meet $\bc_i$ at one of its
endpoints which is impossible as all leaves are pairwise disjoint.
Hence all critical leaves of $\lam$ map to $a$ or $b$. On the other
hand, all leaves whose images are non-disjoint from $a$ or $b$, must be
critical.

It follows from the first and second paragraphs of the proof that the
critical leaves $\bc_1=\ol{x_1y_1}, \dots,
\bc_{d-1}=\ol{x_{d-1}y_{d-1}}$ map alternately to $b$ and $a$ and their
endpoints form the full preimages of $a$ and $b$. Assume that the
positive circular order is given by $b, x_{d-1}, \dots, x_1, a, y_1,
\dots, y_{d-1}$. Then the arc $[x_2, a)$ maps one-to-one onto $\uc$ and
hence has length $\frac{1}{d}$ while the same applies to the arc $(a,
y_2]$. Continuing in the same manner we will see that all critical
leaves are in fact perpendicular to the chord $\ol{ab}$ which in fact
is a diameter of $\disk$. By pulling critical leaves back we complete
the proof in the case when $\si_d(a)=a, \si_d(b)=b$. Other cases can be
considered similarly.
\end{proof}


We are ready to prove the main result of this section.

\begin{proof} [Proof of Theorem \ref{gapinv}:]
Let $G$ be a gap of $\lam$. If there are two adjacent leaves on the
boundary of $G$ which have the same image, then by
Lemma~\ref{colpol} the gap $G$ maps to a leaf. Moreover, if there
are two  leaves in $\bd(G)$ which have the same image and are
connected with a finite concatenation of critical leaves in $\bd(G)$
then by Lemma~\ref{colpol1} again $G$ maps to a leaf. Hence from now
on we may assume that the above two cases do not take place on the
boundary of $G$.

In particular, this implies that the $\si^*_d$-image of the boundary
of $G$ is not an arc. Indeed, otherwise we can choose an endpoint
$x$ of $\si^*_d(\bd(G))$ and a point $\hat x\in \bd(G)$ such that
$\si^*_d(\hat x)=x$. Moreover, clearly $x$ must belong to $\uc$ and
$\hat x$ can be chosen to belong to $\uc$ too. It is easy to see
that under the assumptions made in the first paragraph this is impossible.

Let us show that the $\si^*_d$-image of the boundary of $G$ is
itself the boundary of a gap. To do so, first consider the map $m$
which collapses all critical leaves in $\bd(G)$ to points and is
otherwise one-to-one. Clearly $m(\bd(G))$ is a simple closed curve
and there exists a map $g$ defined on $m(\bd(G))$ such that $g \circ
m = \si^*_d|_{\bd(G)}$. Let us show that $g$ is locally one-to-one.
Clearly $g$ is locally one-to-one on the image of every non-critical
leaf, and by the first paragraph $g$ is locally one-to-one at the
common endpoint of two concatenated leaves in the boundary of $G$.
If $x\in \bd(G)$ is not the endpoint of two concatenated leaves,
then it follows easily from the fact that $\si_d$ is locally
one-to-one that $g$ is locally one-to one at $x$ as well. Set
$m(\bd(G))=S$.

By Lemma~\ref{triorder}, there is no monotone arc $I\subset \bd(G)$
which monotonically maps to \emph{sides} of a monotone triod $T\su
\lam^*$ (as opposed to its central edge, see
Definition~\ref{triod-def}). Let us show that then $\si^*_d(\bd(G))$
is the boundary of a gap. Choose a bounded component $U$ of the
complement of $\si^*_d(\bd(G))$. The boundary $\bd(U)$ of $U$ is a
simple closed curve which contains some leaves. Let $m(\ell)\subset
S$ be the $m$-image of a leaf $\ell\subset \bd(G)$ which maps by $g$
to a leaf $\si_d(\ell)\subset \bd(U)$. Since there exists no
monotone arc $I\subset \bd(G)$ which monotonically maps to
\emph{sides} of a monotone triod $T\su \lam^*$, then, as a point
continues moving along $S$, its $g$-image must move along $\bd(U)$.
Hence $\si^*_d(\bd(G))$ coincides with the simple closed curve
$\bd(U)$.  Moreover, since $g$ is locally one-to-one, $\si^*$ maps
subarcs of $\bd(G)$ monotonically onto subarcs of $\bd(U)$. Hence,
by Lemma~\ref{triorder}, $\bd(U)$ is the boundary of a gap of
$\lam$.

Let a gap $H$ of $\lam$ be such that $\si^*_d(\bd(G))=\bd(H)$. To
show that $\lam$ is gap-invariant it suffices to show that
$\si^*_d|_{\bd(G)}$ can be represented as the positively oriented
composition of a monotone map and a covering map. Consider the map
$m$ which collapses all critical leaves in $\bd(G)$ to points and is
otherwise one-to-one. Clearly, there exists a map $g$ defined on
$m(\bd(G))$ such that $g\circ m=\si^*_d|_{\bd(G)}$. We can define a
circular order on $m(\bd(G))$ by choosing three points $x_i\in
m(\bd(G))$ such that $m^{-1}(x_i)$ is a point. Then $x_1<x_2<x_3$
if and only if $m^{-1}(x_1)<m^{-1}(x_2) <m^{-1}(x_3)$. Let us show
that $g$ preserves orientation.

Consider a point $a\in \partial(G)$ which is not an endpoint of a
critical leaf in $\bd(G)$. By Corollary~\ref{cor-posor1} (if $a$ is an
endpoint of a leaf in $\bd(G)$) or because of the fact that $\si_d$
preserves local orientation (if $a$ is not an endpoint of a leaf from
$\bd(G)$ and is hence approached by points of $\partial(G)$ from the
appropriate side) it follows that $\si^*_d$ (and hence $g$) preserves
the local orientation at all such points $a$.

Let us now assume that $C=\ol{x_1x_2\cdots x_k}\subset \bd(G)$ is a
maximal by inclusion concatenation of critical leaves in $\bd)G)$.
Clearly, $m(C)$ is a point $x$ of $m(\bd(G))$. Choose a monotone arc
$I\subset \bd(G)$ with endpoints $p, q$ such that
$p<x_1<\dots<x_k<q$ and $C\subset I$ (we can make these assumptions
without loss of generality). If neither $x_1$ nor $x_k$ is an
endpoint of a non-critical leaf from $\bd(G)$, it follows from the
properties of $\si_d$ that $g$ preserves orientation. Suppose that
there is a leaf $\ba=\ol{x_ka}\subset I$. Choose a full sibling
collection of $\ol{x_ka}$ and let $\ol{x_1a'}$ be a leaf from this
collection. By Lemma~\ref{loca}, applied repeatedly, $x_1<a'<x_2$.
It is easy to see that the map $\si^*_d$ preserves orientation on
$Q=B\cup \ol{x_1a'}$ where $B$ is a small subarc of $\bd(G)$ with an
endpoint $x_1$ otherwise disjoint from the leaf $\ol{x_1x_2}$ (as
before, in the case when $x_1$ is an endpoint of a leaf
$\ol{x_1b}\subset \bd(G)$ it follows from
Corollary~\ref{cor-posor1}, and in the case when $x_1$ is approached
by points of $\partial(G)$ it follows from the fact that $\si_d$
locally preserves orientation). Therefore $g$ preserves orientation
at $m(x_1)$ as desired.
\end{proof}

The basic property defining $d$-invariant sibling laminations is
that every non-critical leaf can be extended to a collection of $d$
pairwise disjoint leaves with the same image (a full sibling
collection). We conclude this subsection by showing that this
implies the same result for arcs as long as their images are
monotone arcs.

\begin{lem}\label{l:siblingarcs}
Suppose that $\si^*_d$ homemorphically maps an arc $A\su \lam^*$ to
a monotone arc $B\su \lam^*$. Then there are $d$ pairwise disjoint
arcs $A_1=A, A_2, \dots, A_d$ such that for each $i$ the map
$\si^*_d$ homemorphically maps the arc $A_i$ to $B$.
\end{lem}

\begin{proof}
Suppose that the endpoints of $B$ are $p, q$ and that $B\subset [p,
q]=I$. Denote by $J_1, \dots, J_d$ circle arcs which map one-to-one
to $I$. Since $B$ is a monotone arc, there are no more than
countably many leaves in $B$. We can order them so that they form a
sequence of leaves $\ell_n\subset B$ with $\diam(\ell_n)\to 0$.
Given a leaf $\ell_n\subset B$ we can choose its unique preimage
$\ell^1_n\subset A$, and then choose a full sibling collection of
$\ell^1_n$ consisting of leaves $\ell^1_n, \ell^2_n, \dots,
\ell^d_n$. In other words, we choose full preimage collection of
leaves for each leaf from $B$ so that this preimage collection
includes a leaf from $A$. Call $\ell_n$ \emph{long} if there exists
a leaf $\ell^j_n$ with endpoints coming from distinct sets $J_r$ and
$J_s$.

We claim that here are no more than finitely many long leaves in
$B$. Indeed, suppose otherwise. By continuity we can then choose a
subsequence for which we may assume that 1) there is a sequence of
leaves $t_n\subset B$ which converges to a point $x\in \uc$ from one
side, and 2) there is a sequence of their pullback-leaves $\hat t_n$
(i.e., $\si^*_d(\hat t_n)=t_n$) which converges to a critical leaf
$\hat x$ from one side. Clearly, this is impossible since $B$ is a
monotone arc.

Suppose that $\ell_{n_1}=\ol{a_{n_1}b_{n_1}}, \dots,
\ell_{n_k}=\ol{a_{n_k}b_{n_k}}$ are all long leaves in $B$. Without
loss of generality we may assume that $a_{n_1}<b_{n_1}\le
a_{n_2}<\dots <b_{n_k}$. Denote closures of components of $B\sm
\bigcup \ell_{n_i}$ by $S_1, S_2, \dots, S_{k+1}$ numbered in the
natural order on the circle so that $S_1$ precedes $a_{n_1}$, $S_2$
is located between $b_{n_1}$ and $a_{n_2}$, etc (some of these sets
may be empty, e.g. if $p=a_{n_1}$ then $S_1$ is empty).

Then each $S_j$ has $d$ pullbacks, each of which is a monotone arc
$\widehat S^r_j$ such that $\partial(\widehat S^r_j)\subset I_r$.
Let us consider the union of all such pullbacks with all leaves from
previously chosen full preimage collections of all leaves
$\ell_{n_1}, \dots, \ell_{n_k}$ (i.e., of all long leaves in $B$).
Basically, we consider all preimage collections of leaves of $B$ and
then take the closure of their union, representing it in a
convenient form. We claim that each component of this union is a
monotone arc which maps onto $B$ in a one-to-one fashion. Indeed, by
the construction each such component $X$ can be extended both
clockwise and counter-clockwise until it reaches points mapped to
$p$ and $q$ respectively. It follows from the construction that
these components are pairwise disjoint and that one of them is the
originally given arc $A=A_1$. This completes the proof.
\end{proof}

\subsection{The space of all sibling invariant laminations}\label{spasib}

Our approach is to describe laminations in the ``language of
leaves''. The main idea is to use sibling invariant laminations for
that purpose. By Theorem~\ref{gapinv} this does not push us outside
the class of Thurston invariant laminations. According to the
philosophy, explained in the Introduction, now we need to verify if
the  set of all sibling invariant laminations contains all
q-laminations and is Hausdorff closed. The first step here is made
in Lemma~\ref{l-qissib} in which we relate sibling invariant
laminations and q-laminations. Observe that it is well-known (and
easy to prove) that $d$-invariant q-laminations are Thurston
$d$-invariant laminations.

\begin{lem}\label{l-qissib}
$d$-invariant q-laminations are sibling $d$-invariant.
\end{lem}

\begin{proof}
Assume that $\sim$ is a $d$-invariant laminational
equivalence relation. 
Conditions (1) and (2) of Definition~\ref{dfn-invariant} are immediate.
To check condition (3) of Definition~\ref{dfn-invariant}, assume that
$\ell$ is a non-critical leaf of $\lam_\sim$ and verify that there are
$d$ pairwise disjoint leaves $\ell=\ell_1, \dots, \ell_d$ with the same
image. To do so, consider the collection $\A$ of all $\sim$-classes
which map to the $\sim$-class of $\si_d(\ell)$. If a $\sim$-class
$\g\in \A$ does not contain $\ell$ and is not critical, then $\bd(\g)$
contains only a unique sibling of $\ell$. If $\g$ is critical, it maps
to the $\sim$-class of $\si_d(\ell)$, say, $k$-to-$1$, and we can
choose $k$ pairwise disjoint siblings leaves of $\ell$ on the boundary
of $\ch(\g)$. If $\ell$ is an edge of the set $\ch(\h)$ where $\h$ is a
critical class we can still find the appropriate number of its pairwise
disjoint siblings in the boundary of $\ch(\h)$. The endpoints of leaves
from the thus created list exhaust the list of all points which are
preimages of endpoints of $\si_d(\ell)$. Thus, we get $d$ siblings one
of which is $\ell$ as desired.
\end{proof}

We prove that all limits of q-laminations are sibling $d$-invariant.
Thus, if there is a Thurston $d$-invariant lamination $\lam$ which is
not sibling $d$-invariant, then it is not a Hausdorff limit of
q-laminations which shows that the class of Thurston $d$-invariant
laminations is too wide if we are interested in Hausdorff limits of
q-laminations. This justifies our interest in the next example
illustrated on Figure~\ref{nosibf}.

Suppose that points $\hat x_1, \hat y_1, \hat z_1, \hat x_2, \hat
y_2, \hat z_2$ are positively ordered on $\uc$ and $H=\ch(\hat x_1
\hat y_1 \hat z_1 \hat x_2 \hat y_2 \hat z_2)$ is a critical hexagon
of an invariant q-lamination $\lam$ such that $\si^*_d:H\to T$ maps
$H$ in the $2$-to-$1$ fashion onto the triangle $T=\ch(xyz)$ with
$\si_d(\hat x_i)=x, \si_d(\hat y_i)=y, \si_d(\hat z_i)=z$. Now, add
to the lamination $\lam$ the leaves $\ol{\hat x_1 \hat z_1}$ and
$\ol{\hat x_1 \hat x_2}$ and all their pullbacks along the backward
orbit of $H$ under $\si^*_d$. Denote the thus created lamination
$\lam'$. It is easy to see that $\lam'$ is Thurston $d$-invariant
but not sibling $d$-invariant because $\ol{\hat x_1 \hat z_1}=\ell$
cannot be completed to a full sibling collection (clearly, $H$ does
not contain siblings of $\ell$).

\begin{figure}[h!]
\centering\def\svgwidth{.7\columnwidth}\fbox{\fbox{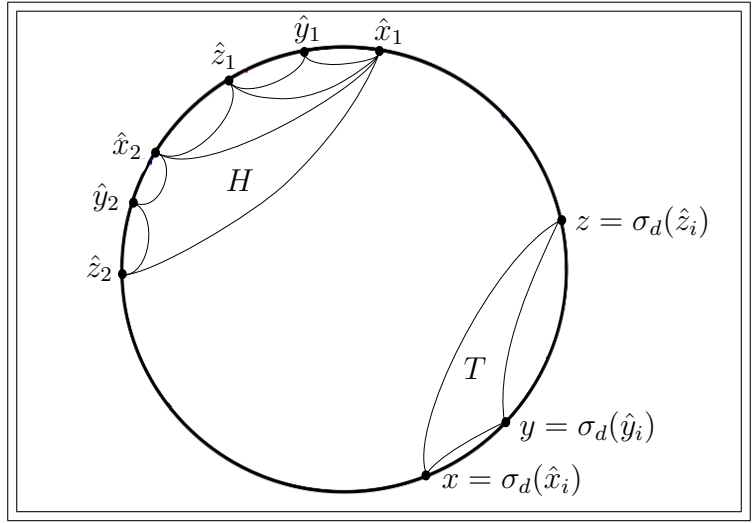}}
 \caption{An example of a Thurston invariant lamination which is not sibling invariant.
 The leaf $\ol{\hat x_1 \hat z_1}$ has no siblings in $H$.}
 \label{nosibf}
\end{figure}

\begin{lem}\label{sibling-sequences}
Take sequences of $d$ sibling leaves $\ol{\ha^j_i\hb^j_i},$ $1\le j\le
d,$ $i=1, 2, \dots$ such that $\ol{\ha^1_i\hb^1_i}\to
\ol{\ha^1\hb^1}=\ell^1,$ $\ol{\ha^2_i\hb^2_i}\to
\ol{\ha^2\hb^2}=\ell^2, \dots,$ $\ol{\ha^d_i\hb^d_i}\to
\ol{\ha^d\hb^d}=\ell^d$ and $\si_d(\ell^1)$ is not degenerate. Then
$\ell^j$, $1\le j\le d$ are siblings with non-degenerate image.
\end{lem}

\begin{proof}
By continuity, $\si_d(\ell^j)=\si_d(\ell^1)$ for all $j$. To show that
leaves $\ell^j, 1\le j\le d$ are pairwise disjoint consider $i_0$ and
$\e>0$ such that for each $i\ge i_0$ and each $j$, $1\le j\le d$ we
have that $|\ol{\ha^j_i\hb^j_i}|\ge \e$. Then it follows that for a
fixed $i$ the pairwise distance between points of sets $\{\ha^j_i,
\hb^j_i\}$, for $1\le j\ne k\le d$ is bounded away from $0$ ($\ha^j_i,
\ha^k_i$ cannot be too close because their images coincide and
$\ha^j_i, \hb^k_i$ cannot be too close because their images are too far
apart). Hence the leaves $\ell^j$, $1\le j\le d$ are disjoint as
desired.
\end{proof}

Suppose that $\lam$ is a prelamination. Then by its closure
$\ol{\lam}$ we mean the set of chords in $\ol{\disk}$ which are
limits of sequences of leaves of $\lam$. It is easy to see that
$\ol{\lam}$ is a closed lamination. Corollary~\ref{c:prelasib} shows
that it is enough to verify the property of being sibling invariant
on dense prelaminations. It immediately follows from
Lemma~\ref{sibling-sequences} (so that it is given here without
proof).

\begin{cor}\label{c:prelasib}
If $\lam$ is a sibling $d$-invariant prelamination, then its closure
$\ol{\lam}$ is a sibling $d$-invariant lamination.
\end{cor}

Theorem~\ref{Thurstonimpliessibling} follows from Lemma~\ref{l-qissib} and
Lemma~\ref{sibling-sequences}.


\begin{thm}\label{Thurstonimpliessibling}
The Hausdorff limit of a sequence of sibling invariant laminations is a
sibling invariant lamination. The space of all sibling invariant
laminations is closed in the Hausdorff sense and contains all
q-laminations.
\end{thm}

\section{Proper laminations}\label{s:prolam}

Clearly, not all (sibling, Thurston) invariant laminations are
q-lami\-na\-tions (e.g., a lamination with two finite gaps with a common
edge is not a q-lamination). In this section we address this issue
and describe Thurston $d$-invariant laminations which almost coincide
with appropriate q-laminations. According to the adopted approach,
we use the ``language of leaves'' in our description.

\begin{dfn}[Proper lamination]
Two leaves with a common endpoint $v$ and the same image which is a
leaf (and not a point) are said to form a \emph{critical wedge} (the
point $v$ then is said to be its vertex). A lamination $\lam$ is
\emph{proper} if it contains no critical leaf with periodic endpoint
and no critical wedge with periodic vertex.
\end{dfn}

Proper laminations are instrumental for our description of
lami-nations which almost coincide with q-laminations.


\begin{lem}\label{l-qprop}
Any q-lamination is proper.
\end{lem}

\begin{proof}
Suppose that $A$ is either a critical wedge or a critical leaf which
contains a periodic point of period $n$. Then $A$ is contained in a
finite class $\g$ such that $|\si_d(\g)|<|\g|$ while on the other
hand $\si_d^n(\g)$ must coincide with $\g$, a contradiction.
\end{proof}

The exact inverse of Lemma~\ref{l-qprop} fails. However it turns out
that proper laminations are very close to q-laminations. To show
this we need a few technical definitions.

\begin{dfn}\label{d:wander}
Suppose that $A$ is a polygon with vertices in $\uc$. It is said to be
\emph{$d$-wandering} if for any $m\ne n$ we have $\ch(\si^m_d(A\cap
\uc))\cap \ch(\si^n_d(A\cap \uc))=\0$.
\end{dfn}

In Definition~\ref{d:wander} we do not require that $A$ be a part of
some lamination or even that the circular orientation of vertices of
$A$ be preserved under $\si_d$. Still, Childers was able to
generalize Kiwi's results \cite{kiw02} and prove in \cite{chi07}
that $A$ cannot have too many vertices.

\begin{thm}[\cite{kiw02, chi07}]\label{t:kiwander}
Suppose that $A$ is a polygon with more than $d^d$ vertices. Then it is
not $d$-wandering.
\end{thm}

In Definition~\ref{d:wander} we assume that images of $A$ have
pairwise disjoint convex hulls. As Lemma~\ref{l:noconcat} shows, one
can slightly weaken this condition and still obtain useful
conclusions.

\begin{lem}\label{l:noconcat}
Suppose that one the following holds.

\begin{enumerate}

\item $A$ is a polygon, $\partial(A)\subset \uc$, and for any $m\ne
    n$ the interiors of the convex hulls $\ch(\si_d^m(A\cap \uc))$ and
    $\ch(\si_d^n(A\cap \uc))$ are disjoint;

\item $A$ is a chord of $\uc$ and for any $m\ne n$ two image chords
    $(\si^*_d)^m(A)$ and $(\si^*_d)^n(A)$ are disjoint inside
    $\disk$.

\end{enumerate}

Then, if $\si_d^n|_{\partial(A)}$ is one-to-one for all $n$, then
either $A$ is wandering, or $A$ is a chord with a preperiodic
endpoint.

\end{lem}

\begin{proof}
Suppose that $A=\ol{pq}$ is a non-wandering chord and $p, q$ have
infinite orbits. We may assume that $\si^*_d(A)=\ol{qr}$ and
$\si_d(p)=q$. Set $A_k=(\si^*_d)^k(A)$. Take closures of the two
components of $\disk\sm A_k$. It follows from the assumptions that
$\bigcup_{i>k} A_i$ is contained in one of them. Hence $A_i$
converge to either a $\si_d$-fixed point on $\uc$ or a
$\si^*_d$-invariant chord. However this contradicts the fact that
$\si_d$ is repelling.

Suppose now that $A$ is an non-wandering polygon. We may assume
that $\si_d^*(A)$ and $A$ intersect either (1) at a common vertex $x$, or
(2) along a common edge $\ell=\ol{xy}$. Choose the vertex $u$ of $A$
with $\si_d(u)=x$ and consider the chord $\ol{ux}$. The chord $\ol{ux}$
satisfies the assumptions of the theorem and is non-wandering. Then by
the above $u$ is preperiodic. We may assume that $u$ is  fixed.
Consider the two edges $\ol{uv_0}$ and $\ol{uw_0}$ of $A$ and the arc
$[v_0,w_0]$ in $\uc$ not containing $u$. Similarly, for each $n$ set
$v_n=\si_d^n(v_0)$ and $w_n=\si_d^n(w_0)$. Since the interiors of the
convex hulls of $\si_d^n(A)$ are pairwise disjoint, the open arcs
$(v_n,w_n)$ are also pairwise disjoint and hence their diameter must
converge to zero, a contradiction with the fact that $\si_d$ is
expanding.
\end{proof}

\begin{lem}\label{simple-wedges}
Suppose that $\lam$ is a Thurston invariant lamination. Then there are
at most finitely many points $x$ such that there is a critical leaf
with an endpoint $x$ or a critical wedge with a vertex $x$.
\end{lem}

\begin{proof}
Clearly there are at most finitely many critical leaves. Hence we
may suppose that there are infinitely many points $x_i$ such that
there are leaves $\ol{a_ix_i}$, $\ol{b_ix_i}$ with
$\si_d(a_i)=\si_d(b_i)\ne \si_d(x_i)$; we may assume that the sets
$\ol{\si_d(a_i)\si_d(x_i)}$ are all distinct and the points
$\si_d(x_i)$ are all distinct. Clearly, the chords $\ol{a_ib_i}$ and
$\ol{a_jb_j}$ are disjoint inside $\disk$. Since each such chord is
critical, we have a contradiction.
\end{proof}

Let $E(v)$ be the set of \emph{all} endpoints of leaves with a common
endpoint $v$ (if $E(v)$ accumulates upon $v$ we include $v$ in
$E(v)$). Then $E(v)$ is a closed set. Let $C(v)$ be the family of
 \emph{all} leaves connecting
$v$ and a point of $E(v)$ (it might include $\{v\}$ as a degenerate
leaf).

\begin{lem}\label{finite-cones}
Suppose that $v$ is a point with infinite orbit and $\lam$ is a
Thurston $d$-invariant lamination. Then there are at most finitely many leaves
with an endpoint $v$.
\end{lem}

\begin{proof}
Let $E(v)$ be infinite. By Lemma~\ref{simple-wedges}, we may assume
that $v$ and all its images are not endpoints of critical leaves or
vertices of critical wedges. If $v$ ever maps to $E(v)$ then by
Lemma~\ref{l:noconcat} $v$ is preperiodic, a contradiction. Choose
$A\subset E(v)$ consisting of $d^d$ points. Consider $\ch(A\cup
\{v\})$. By the above for any $n\ne m$, the interiors of
$\ch(\si_d^n(A\cup \{v\}))$ and $\ch(\si_d^m(A\cup \{v\}))$ are disjoint.
Then by Lemma~\ref{l:noconcat} $\ch(A\cup \{v\})$ is wandering,
contradicting Lemma~\ref{t:kiwander}.
\end{proof}

\begin{lem}\label{noinfcon}
Suppose that $\lam$ is a Thurston $d$-invariant lamination and $A$ is a
concatenation of leaves $A=\bigcup \ol{x_ix_{i+1}}$, $i=0, 1, \dots$,
with $\ol{x_ix_{i+1}}\in\lam$ (the set $\partial(A)$ is infinite).
Then $A$ has preperiodic vertices.
\end{lem}

\begin{proof}
Consider the convex hulls of sets $\ol{\partial(A)}=B_0$,
$\ol{\partial(\si_d(A))}=B_1, \dots$. Suppose that all such convex hulls
have disjoint interiors. There are numbers $n$ such that $\si_d|_{B_n}$
is not one-to-one. This means that there is a critical chord $\ell_n$
inside the convex hull of $B_n$. Since we assume that it is the
interiors of sets $\ch(B_n)$ which are disjoint, one critical chord can
correspond to at most two sets $B_n$; otherwise two critical chords
$\ell_n$ and $\ell_m$ cannot intersect. It follows that there are at
most finitely many critical chords $\ell_i$ constructed as above and
that for large enough $n$ the map $\si_d|_{B_i}$ is one-to-one. By
Lemma~\ref{l:noconcat} this is impossible. Hence convex hulls of sets
$B_n$ have non-disjoint interiors which implies that we can make the
following assumption: there exists $m, n$ such that $\si_d^m(x_0)=x_n$.
We may also assume that $x_n\ne x_0$.

Denote the concatenation of leaves $\ol{x_0x_1}, \dots,
\ol{x_{n-1}x_n}$ by $C$. Then $\si_d(C)$ is a concatenation of leaves
attached to $C$ etc. If for some $k$ we have that $\si_d^k(C)$ is a point
we can choose the minimal such $k$ which implies that $\si_d^{k-1}(C)$ is a
concatenation of leaves whose image is one of its own vertices $y$.
Hence $y$ is periodic as desired. Otherwise we may assume that the
number of vertices of $C$ does not drop under application of the map
$\si_d$. Observe that $C$ may have self intersections. In this case we
may refine $C$ to get a concatenation with no self-intersections still
connecting $x_0$ and $x_n$.

We can optimize the situation even more. Indeed, it is not necessarily
so that $C$ only intersects itself when $\si_d(C)$ gets concatenated to
$C$ at $\si_d(x_0)=x_n$. Thus we may assume that $C$ is the shortest
subchain of leaves in $C$ which ever intersects itself. This implies
that if there are no preperiodic vertices of $C$ then the only way
images of $C$ may intersect is by being concatenated to each other at
their ends. It now follows that $\lim \si_d^n(C)$ is either a leaf in
$\lam$ or a point of $\uc$. In either case this contradicts that $\si_d$
is expanding.
\end{proof}

Recall that $\approx_\lam$ was the equivalence relation
defined by $x{\approx_\lam}y$ if and only if there exists a
\emph{finite} concatenation of leaves of $\lam$ connecting $x$ and
$y$. Theorem~\ref{t:nowander} specifies properties of
$\approx_\lam$.

\begin{thm}\label{t:nowander}
Let $\lam$ be a proper 
Thurston invariant lamination. Then
$\approx_\lam$ is an invariant laminational equivalence relation.
\end{thm}

\begin{proof}
Let us show that any point $v\in \uc$ is the endpoint of at most
finitely many leaves of $\lam$. Otherwise by Lemma~\ref{finite-cones}
we may assume that $v$ is fixed. Take the infinite invariant set
$E'=E(v)\cup\{v\}$. Since $\si_d$ is expanding, $E'$ contains points $x,
x'$ with $\si_d(x)=\si_d(x')$ contradicting the fact that $\lam$ is
proper.

Suppose next that $A$ is an infinite concatenation of leaves $A=\bigcup
\ol{x_ix_{i+1}}$, $i=0, 1, \dots$, with $\ol{x_ix_{i+1}}\in\lam$. By
Lemma~\ref{noinfcon} we may assume that $x_0=x$ is fixed. Let us show
that if $\ell_1=\ol{xe_1},$ $\dots,$ $\ell_k=\ol{xe_k}$ are \emph{all}
the leaves with the endpoint $x$ then $\si_d(e_i)=e_i$ for all $i$. Since
$\lam$ is proper, the leaves $\ell_1, \dots, \ell_k$ have distinct
non-degenerate $\si^*_d$-images. Hence all points $e_i, 1\le i\le k$
are periodic. If $k=1$ then $e_1$ is fixed and we are done. If $k=2$
then $\ell_1, \ell_2$ are edges of some gap $G$ and the fact that the
orientation is preserved under $\si_d$ implies that both $\ell_1,
\ell_2$ are fixed. Suppose that $k\ge 3$. We may assume that
$e_1<e_2<\dots<e_k$. Since gaps map to gaps and the orientation is
preserved on them, the fact that $x$ is fixed implies that then
$\si_d(e_1)<\si_d(e_2)<\dots<\si_d(e_k)$ and hence in fact
$\si_d(e_i)=e_i, 1\le i\le k$. Thus, the leaf $\ol{x_0x_1}$ is fixed,
the leaf $\ol{x_1x_2}$ is fixed and, by induction, all the leaves
$\ol{x_ix_{i+1}}$ are fixed, a contradiction.

By the above $\lam$ contains no infinite cones and no infinite
concatenations of leaves. Let us show that all $\approx_\lam$-classes
are finite (and hence closed). Otherwise let $E$ be an infinite
$\approx_\lam$-class and let $x_0\in E$. For each $y\in E$ fix a
concatenation of leaves $L_y$ from $x_0$ to $y$ containing the least
number of leaves. Then there are infinitely many sets $L_y, y\in E$.
Since there are only finitely many leaves of $\lam$ with an endpoint
$x_0$, we can choose $x_1\in E$ so that there are infinitely many sets
$L_y, y\in E$ whose first leaf is $\ol{x_0x_1}$. Since there only
finitely many leaves of $\lam$ with an endpoint $x_1$ we can choose
$x_2\in E$ so that there are infinitely may sets $L_y, y\in E$ whose
second leaf is $\ol{x_1x_2}$. Continuing in this manner we will find an
infinite concatenation of leaves of $\lam$, a contradiction (by the
choice of sets $L_y, y\in E$ the points $x_0, x_1, \dots$ are all
distinct).

Take convex hulls of $\approx_\lam$-classes. Clearly, these convex
hulls are pairwise disjoint. It follows that if a non-constant sequence
of $\approx_\lam$-classes converges, then it converges to a leaf of
$\lam$ or a point. Hence $\approx_\lam$ is a closed equivalence
relation. To show that $\approx_\lam$ is invariant and laminational we
have to prove that $\approx_\lam$-classes map \emph{onto}
$\approx_\lam$-classes in a covering way (i.e., we need to check
conditions (D1) and (D3) of Definition~\ref{d-lameq}). Let us show that
for any $x\in \uc$ the $\approx_\lam$-class maps onto the
$\approx_\lam$-class of $\si_d(x)$. Indeed, let $y$ belong to the
$\approx_\lam$-class of $\si_d(x)$. Choose a finite concatenation
$\ell_1\ell_2\dots\ell_k$ of leaves connecting $\si_d(x)$ and $y$ (here
$x$ is an endpoint of $\ell_1$ and $y$ is an endpoint of $\ell_k$).
Take a pullback-leaf $\ol{xx_1}$ of $\ell_1$ with an endpoint $x$, then
a pullback-leaf $\ol{x_1x_2}$ of $\ell_2$, etc until we get a finite
concatenation of leaves connecting $x$ and some point $y'$ with
$\si_d(y')=y$. This implies that the $\approx_\lam$-class maps onto the
$\approx_\lam$-class of $\si_d(x)$ as desired.

It remains to prove that $\approx_\lam$ satisfies condition (D3)
from Definition~\ref{d-lameq} (i.e., that $\si$ is covering on
$\approx_\lam$-classes). Observe that if $\lam$ is sibling
invariant, this immediately follows from Corollary~\ref{cor-posor1}.
In the case when $\lam$ is Thurston invariant we need an extra
argument. So, suppose that $\g$ is a $\approx_\lam$-class. Some
edges of $\ch(\g)$ may well be leaves of $\lam$. If $\ell$ is an
edge of $\ch(\g)$ which is \emph{not} in $\lam$ then on the side of
$\ell$, opposite to that where $\g$ is located, there must lie an
infinite gap of $\lam$. It is easy to see that if we now add $\ell$
with its grand orbit, we will get a Thurston invariant lamination.
Hence we may assume from the very beginning, that all edges of
$\ch(\g)$ are leaves of $\lam$.

If $|\g|=2$ then (D3) is automatically satisfied. Otherwise let
$\ell$ be an edge of $\ch(\g)$. Then either (1) $\ell$ is approached
from the outside of $\ch(\g)$ by leaves $\ell_i$ of $\lam$, or (2)
there is an infinite gap $G$ on the other side of $\ell$, opposite
to the side where $\ch(\g)$ is located. Below we refer to these as
cases (1) and (2).

Let us now show that if one of the edges of $\ch(\g)$ is critical,
then all are critical. Indeed, let $\ell=\ol{ab}$ be a critical edge
of $\ch(\g)$. In case (1) images of $\ell_i$ separate $\si_d(\ell)$
from the rest of the circle, hence all points of $\g$ map to the
same point. In case (2) both endpoints of $\ell$ are limit points of
vertices of $G$ because otherwise we could extend the $\approx_\lam$
class $\g$. Since $\lam$ is Thurston invariant we conclude that
$\si_d(\ell)$ is a vertex of an infinite gap $\si_d(G)$, approached
from either side on $\uc$ by vertices of $\si_d(G)$. Hence no leaves
can come out of $\si_d(\ell)$ and again all points of $\g$ map to
the same point. Clearly, in this case (D3) from
Definition~\ref{d-lameq} is satisfied.

Suppose now that all edges of $\ch(\g)$ are non-critical. We claim
that if $\ell$ is an edge of $\ch(\g)$ then $\si_d(\ell)$ is an edge
of $\ch(\si_d(\g))$. Indeed, in the case (1) $\si_d(\ell)$ is
approached by leaves of $\lam$ from one side and in the case (2) it
borders an infinite gap of $\lam$ from one side. In either case it
cannot be a diagonal of the gap $\ch(\si_d(\g))$, and the claim is
proved.

It remains to show that as we walk along the boundary of $\ch(\g)$,
the $\si_d$-image of the point walks in the positive direction along
the boundary of $\ch(\si_d(\g))$. Indeed, suppose first that $\si_d(\g)$
consists of two points. Then by the above there are no critical
edges of $\ch(\g)$, and the condition we want to check is
automatically satisfied. Otherwise let $\ch(\si_d(\g))$ be a gap.
Let $\ol{ab}$ be an edge of $\ch(\g))$ such that moving from $a$ to
$b$ along $\ol{ab}$ takes place in the \emph{positive} direction on
the boundary of $\ch(\g)$. Suppose that moving from $\si_d(a)$ to
$\si_d(b)$ along $\ol{\si_d(a)\si_d(b)}$ takes place in the
\emph{negative} direction on the boundary of $\ch(\si_d(\g))$. Then
the properties of Thurston laminations imply that in the case (1)
images of leaves $\ell_i$ will have to cross $\ch(\si_d(\g))$, a
contradiction. On the other hand, in the case (2) they would imply
that the image of the infinite gap $G$ contains $\ch(\si_d(\g))$, a
contradiction again. Hence the map is positively oriented on
$\bd(\ch(\g))$ as desired.
\end{proof}

Theorem~\ref{t:nowander} shows that, up to a ``finite'' restructuring,
a lamination is a q-lamination if and only it is proper; the
appropriate claim is made in Corollary~\ref{c-propq} whose proof is
left to the reader.

\begin{cor}\label{c-propq}
A proper Thurston invariant lamination $\lam$ is a q-lamination if and only if for each
$\approx_\lam$-class $\g$ the edges of its convex hull $\ch(\g)$ belong
to $\lam$ while no leaf of $\lam$ is contained in the interior of  $\ch(\g)$.
\end{cor}

\section{Clean laminations}

Thurston defined clean laminations. In this section we show that every
clean Thurston invariant lamination is a proper sibling invariant
lamination; thus, up to a minor modification every clean Thurston
invariant lamination is a $q$-lamination. We show in the next section
that every clean Thurston $2$-invariant lamination is a $q$-lamination.

\begin{dfn}\label{def-clean}Let $\lam$ be a lamination. Then $\lam$ is \emph{clean}
if no point of $\uc$ is the common endpoint of three distinct leaves of $\lam$.
\end{dfn}

\begin{thm}\label{tclean}
Let $\lam$ be a Thurston $d$-invariant clean lamination. Then $\lam$
is a proper sibling $d$-invariant lamination.
\end{thm}

\begin{proof}
Let $\lam$ be a clean Thurston $d$-invariant lamination. Suppose
first that $\lam$ contains a critical leaf $\ol{xy}$ with a periodic
endpoint. Assume that $x$ is fixed. Then there must exist $d$
disjoint leaves which map to $\ol{xy}$. One of these must have $x$
as an endpoint. Label this leaf $\ol{xz}$ (since
$\si_d^*(\ol{xy})=x$, $y\ne z$). Similarly there must exist $d$
leaves which map to $\ol{xz}$ one of which must be $\ol{xw}$ (and,
as above, all three leaves are distinct). Hence $\lam$ is not clean,
a contradiction. The case when $\lam$ contains a critical wedge is
similar. Thus, $\lam$ is proper.

Suppose next that $\ell=\ol{xy}\in\lam$ and $\si_d(\ell)$ is
non-degenerate. To show that $\lam$ is sibling $d$-invariant we need to
show that there are $d-1$ siblings of $\ell$. Since $\lam$ is
a Thurston $d$-invariant lamination, there exists a collection $B$ of $d$ pairwise disjoint
leaves $\ell_1,\dots,\ell_d$ so that $\si_d(\ell_i)=\si_d(\ell)$ for
all $i$. If $\ell=\ell_i$ for some $i$ we are done. Otherwise there
exist $i\ne j$ so that $\ell_i\cap\ell\ne \emptyset\ne
\ell_j\cap\ell$. Let $\ell_i=\ol{xz}, \ell_j=\ol{yt}$ and consider two
cases.

(1) Points $z$ and $t$ are located in distinct components of $\uc\sm
\{x, y\}$. Then $\ell_i$ and $\ell$ are edges of a certain gap $G$ because $\lam$ is clean.
Since $\si_d^*|_{\bd(G)}$ is positively oriented in case
$\ch(\si_d(\partial G))$ is a gap, $G$ must be a finite gap of
$\lam$, collapsing to a leaf. Hence there exists an edge of $G$ with an endpoint $y$,
contradicting the assumption that $\lam$ is clean.

(2) Points $z$ and $t$ belong to the same component of $\uc\sm
\{x, y\}$. Similar to (1), there exists a gap $G$ with edges
$\ell_i, \ell, \ell_j$ (and possibly other edges), collapsed onto
$\si_d(\ell)$ under $\si_d$. Since $\lam$ is clean, every leaf of
$\lam$, which intersects $G$, is contained in $\bd(G)$. Hence $\bd(G)$
consists of $2n$ leaves all of which map to $\si_d(\ell)$, and,
possibly, some critical leaves.

Let us show that there are no critical edges of $G$. Suppose that
$\ol{uv}$ is a critical edge of $G$ such that all vertices of $G$ are
contained in the circle arc $I=[v, u]$. Each leaf of $\lam$ close to
$\ol{uv}$ and with endpoint from $(u, v)$ will have the image which
crosses $\si_d(\ell)$. Hence  there are no such leaves and $\ol{uv}$ is
an edge of a gap $H$ whose vertices belong to $[u, v]$. Since $\lam$ is
clean, there are no edges of $H$ through $u$ or $v$ except for
$\ol{uv}$. Hence there exist sequences $u_i\in \partial(H)$ converging
to $u$ and $v_i\in \partial(H)$ converging to $v$. Then points
$\si_d(u_i)$ and $\si_d(v_i)$ are on opposite sides of $\si_d(u)$. It
follows that the leaf $\si_d(\ell)$ cuts the image of $H$, a
contradiction with the assumption that $\lam$ is a Thurston
$d$-invariant lamination. Thus, $\bd(G)$ consists of $2n$ leaves all of
which map to $\si_d(\ell)$.

This implies that in the collection $\{\ell_1, \dots, \ell_d\}=B$
there are exactly $n$ edges of $G$; denote their collection by $A$.
Since $\lam$ is clean, for each $k$ either $\ell_k\cap G=\emptyset$
or $\ell_k\subset \bd(G)$; hence there are $d-n$ leaves in the
collection $\ell_1, \dots, \ell_d$ which are disjoint from $G$. Now,
starting with $\ell$, select $n$ disjoint siblings of $\ell$ from
$\bd(G)$ and unite them with leaves from $B\sm A$ to get a full set
of siblings of $\ell$. As this can be done for any $\ell$, we see
that $\lam$ is sibling $d$-invariant.
\end{proof}

Suppose that $\lam$ is a clean Thurston $d$-invariant lamination and
let $\approx_\lam$ be the equivalence relation defined in
Definition~\ref{d:fineq}; by Theorem~\ref{tclean} $\approx_\lam$ is a
$d$-invariant laminational equivalence relation. By
Corollary~\ref{c-propq}  and since $\lam$ is clean, $\lam$ is a $q$-lamination if and only if
every chord in the boundary of the convex hull of an equivalence class
of $\approx_\lam$ is a leaf of $\lam$.  We further study the possible difference between the two laminations.
For an equivalence class $\g$, denote by $A_\g$ the union of all leaves
of $\lam$ which join points of $\g$. Since $\lam$ is clean, each $A_\g$ is either
a point, a simple closed curve, a single leaf, or a an arc which contains at least two leaves.
In all but the last case all leaves of $\lam$ which are contained in $A_\g$ are also leaves
of $\lam_{\approx_\lam}$.  It follows that $[\lam\sm \lam_{\approx_\lam}] \cup [\lam_{\approx_\lam}
\sm \lam]$ is contained in the countable union of the convex hulls of equivalence classes
$\g_i$ so that $A_{\g_i}$ is an arc containing at least two leaves. We further
specify this set in
Corollary~\ref{c-primq}.


\begin{cor}\label{c-primq}
Let $\lam$ be a clean Thurston $d$-invariant lamination and  $\g$
an equivalence class of $\approx_\lam$ such that $A_\g$ is an arc
which contains at least two leaves of $\lam$. Suppose that
$\ol{ab}\subset \ch(\g)$. Then if
$\ell=\ol{ab}\in\lam_{\approx_\lam}\sm\lam$, then there exists an
infinite gap $U$ of $\lam$ so that $\ell\sm\{a,b\}$ is contained in
the interior of $U$ and the subarc of $A$ which connects $a$ and $b$
is a maximal concatenation of leaves in $\bd(U)$. Vice versa, if
$\ell=\ol{ab}\in \lam \sm \lam_{\approx_\lam}$, then
$\ell\sm\{a,b\}$  is contained in the interior of $\ch(\g)$ and
$\ell$ is the intersection of two infinite gaps of $\lam$.
\end{cor}

\begin{proof}
Suppose that $\ol{ab}\subset \ch(\g)$. and
$\ell=\ol{ab}\in\lam_{\approx_\lam}\sm\lam$. Since $\g$ is finite,
no leaf of $\lam$ can intersect the chord $\ell$ inside $\disk$ and
there exists a gap $U$ of $\lam$ such that $\ell\sm\{a,b\}$ is
contained in the interior of $U$. If $U$ is finite, then
$\bd(U)\subset A_\g$, a contradiction. Since $\lam$ is clean, the
subarc $[a,b]_{A_{\g}}$ of $A_\g$ is contained in the boundary of
$U$. Moreover, since $\ell$ is an edge of $\ch(\g)$,
$[a,b]_{A_{\g}}$ is a maximal concatenation of leaves in $\bd(U)$.

Conversely, suppose that $\ell=\ol{ab}\in \lam \sm
\lam_{\approx_\lam}$. Then $\ell\sm \{a,b\}$ is contained in the
interior of $\ch(\g)$. Hence $\ell$ is isolated and there exist two
gaps $U,V$ of $\lam$ so that $\ell=U\cap V$. If one of these gaps is
finite, then its boundary is a subset of $A_\g$, a contradiction.
\end{proof}

\section{Quadratic invariant laminations}

In this section we study quadratic laminations. First we show that
Corollary~\ref{c-propq} can be made more precise in the quadratic case.

If a $2$-invariant q-lamination $\lam$ has a finite critical gap $L$ then one can
insert a critical diameter  connecting two vertices of $L$ and then
pull it back along the backward orbit of $L$. Also, if $L$ has six
vertices or more, one can insert a critical (collapsing)
quadrilateral inside $L$ and then pull it back along the backward
orbit of $L$; one can also insert in $L$ a quadrilateral which
itself splits into two triangles by a diameter and then pull it back
along the backward orbit of $L$. In this way one can create  proper
sibling invariant laminations which are not q-laminations. In fact,
a lamination may already exhibit the above described phenomena.
Thus, if a lamination contains a finite critical polygon $L$ which
contains a critical leaf (collapsing quadrilateral) in the interior
of its convex hull, then we say that it has a \emph{critical
splitting by a leaf} (resp. \emph{quadrilateral}).
Corollary~\ref{c-quadprop} shows that these two mechanisms are
\emph{the only} ways a proper quadratic lamination can be a
non-q-lamination.

\begin{cor}\label{c-quadprop}
A quadratic sibling invariant lamination is a q-lamination if and only if it is proper
and does not have a critical leaf (quadrilateral) splitting.
\end{cor}

\begin{proof}
Clearly every $q$-lamination is proper and has no critical splitting
(leaf or quadrilateral). Assume next that $\lam$ is a proper sibling
invariant lamination which does not have a critical splitting (leaf or
quadrilateral). Define $\approx_\lam$ as in Definition~\ref{d:fineq}.
Let us show that for each $\approx_\lam$-class $\g$ the edges of its
convex hull $\ch(\g)$ belong to $\lam$. Suppose that for a
$\approx_\lam$-class $\g$ there is an edge of $\ch(\g)$ not included in
$\lam$. By definition, there are finite concatenations of edges of
$\lam$, connecting all points of $\g$. Hence $\ch(\g)$ cannot be a
leaf and  $\g$ consists of more than two points. Then by
Thurston's No Wandering Triangles Theorem \cite{thu09} $\g$ is either
(pre)periodic or (pre)critical (observe that $\g$ can first map into a
critical class of $\approx_\lam$ and then into a periodic class of
$\approx_\lam$, but not vice versa because $\lam$ is proper).

Consider cases. Suppose that $\g$ is (pre)periodic but not
(pre)critical. Then for some $n$ the $\approx_\lam$-class $\si_2^n(\g)$
is periodic. By an important result of \cite{thu09} the edges of
$\ch(\si_2^n(\g))$ form one periodic orbit of edges. Since at least one of
them is in $\lam$, they all are in $\lam$. Since $\g$ maps onto
$\si_2^n(\g)$ one-to-one by our assumptions, and because $\lam$ is a
sibling (and hence, by Theorem~\ref{gapinv}, a Thurston) invariant
lamination, then all edges of $\ch(\g)$ are in $\lam$ as desired.

Now, suppose that $\g$ is precritical and $\si_2^n(\g)$ is critical.
Again, we may assume that $\ch(\g)$ is not a leaf. Since $\si_2^n(\g)$ is a
critical $\approx_\lam$-class, it must have $2k$-edges and must map
onto its image two-to-one. It follows that the edges of $\si_2^n(\g)$
are limits of sequences of $\approx_\lam$-classes. Indeed, otherwise
there are gaps of $\approx_\lam$ sharing common edges with
$\si_2^n(\g)$. By construction this would mean that these gaps are
infinite and hence a forward image of one of these gaps is a critical
$\approx_\lam$-class. Since we deal with quadratic laminations and $\g$
is also critical, it is easy to see that this is impossible. Thus, the
edges of $\si_2^n(\g)$ are limits of sequences of
$\approx_\lam$-classes which implies that edges of $\si_2^n(\g)$ are
leaves of $\lam$. As before, since $\lam$ is a Thurston invariant
lamination, then all edges of $\g$ are leaves of $\lam$. Thus, in any
case if $\g$ is a $\approx_\lam$-class then its edges are leaves of
$\lam$.

It remains to show that $\ch(\g)$ cannot contain any leaves of $\lam$
in its interior. Indeed, suppose otherwise. We may assume that $\g$ has at least
4 vertices. Suppose that $\g$ is (pre)critical and $\si_2^n(\g)$ is
critical. Let us show that any leaf inside $\si_2^n(\g)$ must have the
image which is an edge or a vertex of $\si_2^{n+1}(\g)$. Indeed, it
suffices to consider the case when $\si_2^n(\g)$ has at least six vertices
and $\si_2^{n+1}(\g)$ is a gap. By No Wandering Triangles Theorem
\cite{thu09} it is (pre)periodic and $\si_2^{n+m}(\g)$ is periodic. By the
above quoted result of \cite{thu09} the edges of $\ch(\si_2^{n+m}(\g))$
form one periodic orbit of edges. Hence if there is a leaf of $\lam$
inside $\ch(\si_2^{n+m}(\g))$, it will cross itself under the appropriate
power of $\si_2$, a contradiction. Thus, any leaf inside $\si_2^n(\g)$
must have the image which is an edge or a vertex of $\si_2^{n+1}(\g)$.

We show next that such a leaf cannot exist. In other words,
since $\lam$ does not admit a critical leaf (quadrilateral) splitting,
we need to show that no other splitting of $\ch(\si_2^n(\g))$ by leaves of
$\lam$ is possible either. Indeed, suppose that there are leaves of
$\lam$ inside $\ch(\si_2^n(\g))$. It cannot be just one critical leaf as
then $\lam$ would admit a critical leaf splitting. Neither can it be a
quadrilateral or a quadrilateral with a critical leaf inside (because
$\lam$ does not admit a critical quadrilateral splitting). Now, suppose
that there is a unique leaf $\ell$ of $\lam$ inside $\si_2^n(\g)$ such
that $\si_2(\ell)$ is an edge of $\si_2^{n+1}(\g)$. Then it has to have a
sibling leaf which will also be a leaf inside $\si_2^n(\g)$. Hence $\si_2^n(\g)$
contains a collapsing quadrilateral, a
contradiction. As these possibilities exhaust all possibilities for
leaves inside $\ch(\si_2^n(\g))$, it follows that there are no leaves
inside $\ch(\si_2^n(\g))$ and hence no leaves inside $\ch(\g)$ as desired.
\end{proof}

\begin{cor}
Suppose that $\lam$ is a clean Thurston $2$-invariant lamination.
Then $\lam$ is a  $q$-lamination.
\end{cor}

\begin{proof}
Suppose that $\lam$ is a clean, Thurston $2$-invariant lamination.
By Theorem~\ref{tclean}, $\lam$ is proper and sibling invariant.
Moreover, since $\lam$ is clean, it does not have a critical leaf
(quadrilateral) splitting. Hence the result follows from
Corollary~\ref{c-quadprop}.
\end{proof}

\bibliographystyle{amsalpha}

\end{document}